\documentclass[12pt]{amsart}
\usepackage[utf8]{inputenc}
\usepackage{mathrsfs}
\usepackage{tikz-cd}
\usepackage[all]{xy}
\usepackage{enumitem}
\usepackage[all]{xypic}
\usepackage{amssymb, amsmath, amscd, amsthm,mathtools}

\usepackage{epstopdf}
\usepackage{mathdots}
\usepackage[bookmarks=true,bookmarksnumbered=true]{hyperref}
\usepackage[hmargin=2cm,vmargin=3cm]{geometry}
\usepackage{ulem}

\allowdisplaybreaks

\numberwithin{equation}{section}

\theoremstyle{plain}
\newtheorem{thm}{Theorem}[section]
\newtheorem{lem}[thm]{Lemma}
\newtheorem{prop}[thm]{Proposition}
\newtheorem{cor}[thm]{Corollary}

\newtheorem*{hyp}{Working Hypothesis}
\newtheorem{hyp1}{Working Hypothesis}
\newcommand{\Mp}{\widetilde{\Sp}}

\theoremstyle{definition}
\newtheorem{rem}[thm]{Remark}

\newcommand{\cl}[1]{\widetilde{#1}}

\newcommand{\pair}[1]{\langle #1 \rangle}

\newcommand{\tr}{\mathrm{tr}}

\def\cha{\operatorname{char}}

\def\Hom{\operatorname{Hom}}
\def\ind{\operatorname{ind}}
\def\Ind{\operatorname{Ind}}
\def\Irr{\operatorname{Irr}}

\def\Re{\operatorname{Re}}

\def\unit{\operatorname{unit}}
\def\temp{\operatorname{temp}}
\def\supc{\operatorname{supc}}
\def\gen{\operatorname{gen}}

\def\beq{\begin{equation}}
\def\eeq{\end{equation}}
\def\beqn{\begin{equation*}}
\def\eeqn{\end{equation*}}
\def\beqna{\begin{eqnarray}}
\def\eeqna{\end{eqnarray}}
\def\beqnan{\begin{eqnarray*}}
\def\eeqnan{\end{eqnarray*}}

\def\wt{\widetilde}

\def\rm{\mathrm}
\def\mc{\mathcal}

\def\mf{\mathfrak}

\def\bs{\backslash}

\def\GL{\mathrm{GL}}

\def\O{\mathrm{O}}
\def\SO{\mathrm{SO}}

\def\Sp{\mathrm{Sp}}
\def\G{\mathrm{G}}
\def\H{\mathrm{H}}
\def\J{\mathrm{J}}
\def\U{\mathrm{U}}
\def\V{\mathrm{V}}
\def\P{\mathrm{P}}
\def\B{\mathrm{B}}

\def\N{\mathrm{N}}
\def\M{\mathrm{M}}
\def\T{\mathrm{T}}
\def\Q{\mathrm{Q}}
\def\A{\mathbb{A}}
\def\Z{\mathrm{Z}}

\def\e{\epsilon}
\def\AA{\mathbb{A}}
\def\CC{\mathbb{C}}

\def\FF{\mathbb{F}}
\def\VV{\mathbb{V}}
\def\WW{\mathbb{W}}

\def\RR{\mathbb{R}}
\def\QQ{\mathbb{Q}}

\def\ss{\subset}

\def\la{\langle}
\def\ra{\rangle}

\def\bs{\backslash}

\title{The local converse theorem for $\Mp_{2n}$}
\author{Jaeho Haan}


\address{Department of Mathematical Sciences, KAIST, 291 Daehak-ro, Yuseong-gu, Daejeon, 34141, South Korea}
\email{jaehohaan@gmail.com}
\begin{document}

\maketitle

\begin{abstract}In this paper, we establish the local converse theorem and the stability of local gamma factors for $\Mp_{2n}$ via the precise local theta correspondence between $\Mp_{2n}$ and $\SO_{2n+1}$ over local fields of characteristic not equal to 2. We also prove the rigidity theorem for irreducible generic cuspidal automorphic representations of $\Mp_{2n}$ over number fields.
\end{abstract}

\section{\textbf{Introduction}} 

Let \(F\) be a non-archimedean local field of characteristic not equal to 2 (i.e., a finite extension of \(\QQ_p\) for any odd prime \(p\) or \(\FF_q((t))\) for  \(q = p^k\)), and let \(W\) be a \(2n\)-dimensional symplectic space over \(F\). Define \(\wt{\Sp}_{2n} = \Mp(W)\) as the metaplectic group of \(W\), which is the double covering group of \( \Sp(W)\), the isometry group of \(W\).

We introduce \(\psi\), a non-trivial additive character of \(F\). For \(\lambda \in F^{\times}\), we define the character \(\psi_{\lambda}\) of \(F\) as \(\psi_{\lambda}(x) \coloneqq \psi(\lambda x)\). Let \(\U'\) be the unipotent radical of a Borel subgroup of \(\Sp(W)\), and let \(\mu_{\lambda}'\) be a certain non-trivial character of \(\U'(F)\) that depends on \(\psi\) and \(\lambda\) (see Section \ref{gen} for its definition). 
In this paper, we prove the following local converse theorem for quasi-split metaplectic groups under the following hypothesis in $\cha(F)=p$ case. (There is no hypothesis in $\cha(F)=0$ case.)

\begin{hyp}
The $\gamma$-factors for generic representations of $\wt{\Sp}_{2n}  \times \GL_l$ are properly defined in $\cha(F)=p$ cases. Furthermore, they satisfy natural properties of $\gamma$-factors. (For the precision of the natural properties, see (i)-(v) in Proposition~\ref{gamma} )
\end{hyp}

\begin{thm}[Local Converse Theorem for $\wt{\Sp}_{2n}$] \label{main} Assume the above working hypothesis on $\cha(F)=p$ case. Let $\wt{\pi}$ and $\wt{\pi}'$ be irreducible admissible $\mu_{\lambda}'$-generic representations of $\wt{\Sp}_{2n}(F)$, satisfying the condition
\[\gamma(s,\wt{\pi} \times \sigma,\psi_{\lambda})=\gamma(s,\wt{\pi}' \times \sigma,\psi_{\lambda})
\] holds for all irreducible supercuspidal representation $\sigma$ of $\GL_i(F)$ with $1\le i \le n$. Then, it follows that
\[\wt{\pi} \simeq \wt{\pi}'.
\]
\end{thm}
\noindent 
Here, $\gamma$ denotes the local gamma factor (see Section 2.3 for a detailed explanation). The local converse theorem (LCT) has been studied in various versions concerning different groups. We provide a brief overview of some notable works related to this theorem.

The initial investigation of these types of problems can be attributed to Henniart \cite{He}, who established a (weak) version of the theorem for general linear groups. Jiang and Soudry \cite{JS03}, building upon Henniart's work, proved a (weak) LCT for odd special orthogonal groups $\SO_{2n+1}$. They combined the global weak functorial lift from $\SO_{2n+1}$ to $\GL_{2n}$, the local descent from $\GL_{2n}$ to $\Mp_{2n}$, and the Howe lifts from $\SO_{2n+1}$ to $\Mp_{2n}$. Their proof involved a global-to-local argument that employed Henniart's result. Subsequently, D. Jiang \cite{Jng} proposed a version of the LCT for all classical groups.

Around 2017, Chai \cite{Cha19} and Jacquet and Liu \cite{JL18} independently refined Henniart's result to its optimal form. More recently, P. Yan and Q. Zhang \cite{YZ23} achieved the same result through a different approach.

Recently, Q. Zhang \cite{Q18, Q19} proved the supercuspidal cases of the conjecture for $\Sp_{2n}$ and $\U_{2n+1}$ using the theory of partial Bessel functions developed by Cogdell, Shahidi, and Tsai in \cite{CST17}. In a similar vein, B. Liu and A. Hazeltine \cite{LH22} established the conjecture for split even special orthogonal groups, while Y. Jo \cite{Jo} refined Jiang and Soudry's \cite{JS03} result and extended Zhang's \cite{Q18} result to the generic cases. Notably, Jo's work addressed local fields of positive characteristics as well as characteristic zero, in contrast to other LCT results that only dealt with characteristic zero cases. These proofs were purely local and did not invoke the weak Langlands functorial lifts to general linear groups, as in the recent results of \cite{Cha19} and \cite{JL18}. For other classical groups, Morimoto \cite{M} demonstrated the theorem for even unitary groups by pulling back the recent results of \cite{Cha19} and \cite{JL18} using global (and hence) local descent methods.

With regards to the LCT for metaplectic groups, it appears plausible to apply the local descent methods by pulling back the LCT for general linear groups (\cite{Cha19} and \cite{JL18}). Alternatively, the theory of partial Bessel functions \cite{CST17} used in \cite{Q18, Q19, LH22} and \cite{Jo} can be employed. Both approaches would require significant effort. In this paper, we propose an alternative approach using the Local Theta Correspondence (LTC) between $\SO_{2n+1}$ and $\Mp_{2n}$, which builds on the ideas of Jiang–Soudry \cite{JS03}. The LTC allows for the transfer of problems between different groups, similar to how Gan and Ichino used it to prove the local Gan-Gross-Prasad conjecture for skew-hermitian unitary groups (\cite{GI2}) and the endoscopic classification of automorphic representations of metaplectic groups (\cite{GI3}.) Inspired by these works, we aim to transfer the LCT for $\SO_{2n+1}$ to $\Mp_{2n}$ using the LTC. 

To successfully transfer the LCT, it is necessary to demonstrate that the LTC preserves certain properties and invariants of representations. Specifically, we consider genericity (Theorem~\ref{Theta1}) and $\gamma$-factors (Theorem~\ref{Theta}). Indeed, these results can be found in \cite[Theorem~9.1]{GS} and \cite[Proposition~11.1, Corollary~11.2]{GS}, respectively.

However, unfortunately, the proof of \cite[Theorem~9.1]{GS} is absent in the cited reference. To address this gap, we provide a complete proof demonstrating that the missing argument holds even in the case when $\cha(F)=p$. Additionally, \cite[Proposition~11.1]{GS} relies on the general $\gamma$-factor theory, which is currently unavailable. Consequently, we furnish a comprehensive proof for this proposition using only the established theory on generic $\gamma$-factor  (see Remark~\ref{err}).

Theorem~\ref{Theta1} and Theorem~\ref{Theta} form the main part of the paper and are discussed in Section 3, along with Section 4, which computes the twisted Jacquet module of the Weil representation—a crucial step in proving the preservation of genericity.

In summary, while Jiang–Soudry already demonstrated the power of the LTC in the proof of the LCT for $\SO_{2n+1}$, our contribution is to adapt and extend this approach to the metaplectic group $\Mp_{2n}$. This approach illustrates how the LTC can be used to transfer the LCT between different groups. Notably, a similar method could be applied to extend the recent result on the LCT for split $\SO_{2n}$ to quasi-split $\SO_{2n}$. However, due to potential complications in notation and subtle differences in the proof and result, we chose to address this extension separately in \cite{HKK23}. It is worth mentioning that, while the main theorem in \cite{HKK23} holds in the case $\mathrm{char}(F) = p$ under assumptions regarding the general theory of $\gamma$-factors (including non-generic cases), the main theorem in this paper holds in the case $\mathrm{char}(F) = p$ under assumptions restricted to the theory of generic $\gamma$-factors.

It's worth noting that in the case where the local field $F$ has characteristic 0, the main theorem (Theorem~\ref{main}) can be readily derived by using Gan-Ichino and Arthur's results concerning the local Langlands correspondence for $\Mp_{2n}$ (\cite{GI3}, \cite{Ar13}), in conjunction with the LCT for $\GL_{2n}$ and the uniqueness of generic members in $L$-packets (\cite{At}). However, we stress that our approach is entirely independent of Arthur's results. We believe that a proof not relying on Arthur's results holds its intrinsic value.

In addition to the main theorem (Theorem~\ref{main}), the paper also proves the following stability of local gamma factors for $\Mp_{2n}$ as a byproduct of the proof. 

\begin{thm} [Stability of local $\gamma$-factors for $\Mp_{2n}$]\label{stab}
Let $\wt{\pi}$ and $\wt{\pi}'$ be irreducible admissible $\mu_{\lambda}'$-generic representations of $\Mp_{2n}(F)$. Then there exists $l=l(\wt{\pi},\wt{\pi}')$ such that for any quasi-character $\chi$ of $F^{\times}$, with conductor $\text{cond}(\chi)>l$, we have
\[ \gamma(s,\wt{\pi} \times \chi,\psi_{\lambda})=\gamma(s,\wt{\pi}' \times \chi,\psi_{\lambda}).\]
\end{thm}

When the residual characteristic of the $p$-adic field $F$ is odd, Zhang (\cite{Q17}) directly proved this result under the condition that $\wt{\pi}$ and $\wt{\pi}'$ share the same central characters. Our result covers cases where the residual characteristic of $F$ is 2 without requiring the assumption of identical central characters. Alternatively, we can remove the assumption on the same central characters from the local functional equation of doubling zeta integrals of metaplectic groups (see \cite[Section~4.3]{GS}).

The paper also presents the rigidity theorem for $\Mp_{2n}$ over a global number field. 

\begin{thm} [Rigidity theorem for $\Mp_{2n}$] \label{rigid}
Let $K$ be a global number field and $\AA$ its ad\`{e}le ring. Let $\mf{U}'$ be a non-trivial generic character of $\U'(K)\bs \U'(\A)$ and $\wt{\pi}_1=\otimes_v \wt{\pi}_{1,v}$ and $\wt{\pi}_2=\otimes_v \wt{\pi}_{2,v}$ be irreducible cuspidal $\mf{U}'$-generic automorphic representations of $\Mp_{2n}(\A)$. If $\wt{\pi}_{1,v}\simeq \wt{\pi}_{2,v}$ for almost all places $v$ of $K$, then $\wt{\pi}_{1,v} \simeq \wt{\pi}_{2,v}$ for all places $v$ of $K$.
\end{thm}

The rigidity theorem states that if two irreducible cuspidal $\mf{U}'$-generic automorphic representations $\wt{\pi}_1$ and $\wt{\pi}_2$ of $\Mp_{2n}(\mathbb{A})$ are equivalent at almost all places, then $\wt{\pi}_1$ is equivalent to $\wt{\pi}_2$.

The paper is organized as follows:

In Section 2, we provide the necessary background and set up for the precise formulation of the theorems. 
We also reduce Theorem~\ref{main} to its weak version, which is stated as Theorem~\ref{main1}.
In Section 3, we explain the LTC and review some of its known properties. We then proceed to prove that the LTC preserves genericity and gamma factors between $\Mp_{2n}$ and $\SO_{2n+1}$. This is a key component of the results presented in the paper.
In Section 4, we compute the twisted Jacquet module of the Weil representation, which plays a crucial role in proving the preservation of genericity in Section 3.
In Section 5, we provide a proof of Theorem~\ref{main} and Theorem~\ref{stab}.
In Section 6, we prove the rigidity theorem for $\Mp_{2n}$.

\subsection{Notations}
\begin{itemize}
\item $F$ : a non-archimedean local field of characteristic different from 2
\item $\cha(F)$ : the characteristic of $F$
\item $q$ : the order of the residue field of $F$ such that $q=p^e$ for some odd prime $p$
\item $\psi$ : a non-trivial additive character of $F$
\item $\varpi$ : a uniformizer of $F$
\item $|\cdot|_F$ : a normalized absolute value in $F$ such that $|\varpi|=q^{-1}$
\item $\Irr(G)$ : the set of isomorphism classes of irreducible representations of $G$, where $G:=\G(F)$ is a $F$-points of a reductive group $\G$ defined over $F$.
\item $\Irr_{\temp}(G)$ : the set of isomorphism classes of irreducible tempered representations of $G$
\item $\Irr_{\supc}(G)$ : the set of isomorphism classes of irreducible supercuspidal representations of $G$
\item $\Irr_{\unit}(G)$ : the set of isomorphism classes of irreducible unitary representations of $G$

\item $\pi^{\vee}$ : the contragredient representation of $\pi \in \Irr(G)$
\item $V_m$ : a $2m+1$-dimensional orthogonal space, i.e., a $2m+1$-dimensional vector space over $F$ equipped with a non-degenerate symmetric bilinear form
\item $\H_m:=\SO(V_m)$ : the special orthogonal group of $V_m$ associated to a non-degenerate symmetric bilinear form $( \ , \ )_{V_m}$
\item $W_n$ : a $2n$-dimensional symplectic space over $F$, i.e., a $2n$-dimensional vector space over $F$ equipped with a non-degenerate symplectic bilinear form $\la \ , \ \ra_{W_n}$

\item $\J_n:=\Sp(W_n)$ : the symplectic group of $W_n$
\item $\wt{\J_n}:=\Mp(W_n)$ : the metaplectic group of $W_n$

\item $\mathbb{I}$ : the trivial representation 
%\item $N_{E/F}$ : a norm map from $E$ to $F$, where $E$ is a quadratic extension of $F$
%\item $\Ind_{B}^G$ : the unnormalized induction for an algebraic group $G$ and its closed subgroup $B$
\item $\ind_{B}^G$ : the unnormalized compactly supported induction for an algebraic group $G$ and its closed subgroup $B$
\item $\mc{S}(X)$ : the space of locally constant function of compact support on a set $X$

\end{itemize}
\newpage
\section{Preliminaries}
In this section, we prepare the basic setup to state our main theorem. While many of the theorems and propositions we quote in this paper only address the case where $\cha(F)=0$, most of the proofs apply equally to the positive characteristic cases. Therefore, instead of repeating the same proof in the references for the case $\cha(F)=0$, we adopt the position that they also cover the case $\cha(F)=p$ when it is applicable. However, when the proof of a proposition or a theorem for $\cha(F)=0$ is not clearly stated in the references, we provide a proof that covers both cases. Furthermore, when the proofs for $\cha(F)=0$ and $
\cha(F)=p$ are a bit different, we mention such differences in the proof.
\subsection{Metaplectic and orthogonal groups}
\subsubsection{Symplectic groups}
Let \(W_n\) be a \(2n\)-dimensional vector space over the field \(F\) equipped with a non-degenerate symplectic form denoted by \(\langle \cdot, \cdot \rangle_{W_n}\). We denote the associated isometric group as \(\Sp(W_n)\).

We choose a specific basis \(\{f_1,\cdots,f_n,f_n^*,\cdots,f_1^*\}\) for \(W_n\) such that the following relations hold:
\[
\langle f_i,f_j \rangle_{W_n} = \langle f_i^*,f_j^* \rangle_{W_n} = 0, \quad \langle f_i,f_j^* \rangle_{W_n} = \delta_{ij},
\]
where \(\delta_{ij}\) is the Kronecker delta.

For \(1 \leq k \leq n\), we define
\[
X_k = \text{Span}\{f_1,\cdots,f_k\} \quad \text{and} \quad X_k^* = \text{Span}\{f_1^*,\cdots,f_k^*\}.
\]
This gives us \(W_n = X_n \oplus X_n^*\).

We also define
\[
W_{n,k} = \text{Span}\{f_{k+1},\cdots,f_n,f_n^*,\cdots,f_{k+1}^*\},
\]
so that \(W_n = X_k \oplus W_{n,k} \oplus X_k^*\).

Next, we consider the flag of isotropic subspaces
\[
X_{k_1} \subset X_{k_1+k_2} \subset \cdots \subset X_{k_1+\cdots +k_r} \subset W_n.
\]
The stabilizer of such a flag is a parabolic subgroup \(\P'\) of \(\Sp(W_n)\) whose Levi factor \(\M'\) is isomorphic to
\[
\M' \simeq \GL_{k_1} \times \cdots \times \GL_{k_r} \times \Sp(W_{n,k_1+\cdots + k_r}),
\]
where each \(\GL_{k_i}\) is the group of invertible linear maps on \(\text{Span}\{f_{\sum_{t=1}^{i-1}k_t+1},\cdots,f_{\sum_{t=1}^{i}k_t}\}\).

\subsubsection{Metaplectic groups}
For a symplectic space $W_n$, metaplectic group $\Mp(W_n)$ is a two-fold cover of $\Sp(W_n)$ defined by
\[
\Mp(W_n)=\Sp(W_n)\times \{\pm{1}\}
\] with group law given by
\[
(g_1,\e_1) \times (g_2,\e_2)=(g_1g_2,\e_1\e_2\cdot c(g_1,g_2)),
\] where $c$ is the Rao's 2-cocycle on $\Sp(W_n)$, whose values lie in $\{\pm1\}$. (For its definition, see \cite{Rao93}.)
The covering group $\wt{\GL}(X_k)$ of $\GL(X_k)$ is defined by $\GL(X_k) \times \{\pm1\}$ with group law
\[
(g_1,\e_1)\cdot(g_2,\e_2)=(g_1g_2,\e_1\e_2\cdot(\det g_1, \det g_2)), \quad \quad g_1,g_2 \in \GL(X_k),
\]
where $(\cdot, \cdot)$ denotes the Hilbert symbol. 

For a Levi subgroup $\M' \simeq \GL_{k_1} \times \cdots \times \GL_{k_r} \times \Mp(W_{0})$ of $\Sp(W_n)$, denote its double covering group $\wt{\M}'$ by
\[
\wt{\M}' \simeq \wt{\GL_{k_1}} \times_{ \{\pm{1}\} } \cdots \times_{\{\pm{1}\}} \wt{\GL_{k_r}} \times_{\{\pm{1}\}} \Mp(W_{0}).
\]
Any unipotent subgroup $\N'$ of $\Sp(W_n)$ splits uniquely in $\Mp(W_n)$. Therefore, we denote its splitting image in $\Mp(W_n)$ using the same symbol $\N'$. If $\P'=\M'\N'$ is a parabolic subgroup of $\Sp(W_n)$, we call $\wt{\P}'=\wt{\M}'\N'$ the double covering group of $\P'$. Every parabolic subgroup of $\Mp(W_n)$ has this form.

\subsubsection{Orthogonal groups}
Let $V_m$ be a $(2m+1)$-dimensional vector space over $F$ equipped with a non-degenerate symmetric bilinear form $(\cdot,\cdot)_{V_m}$ of discriminant 1. There are two isomorphism classes of $V_m$ depending on the dimension of maximal isotropic subspaces of $V_m$. If it is $m$, it is denoted by $V_m^{+}$ and called a split quadratic space. Otherwise, the dimension of maximal isotropic subspaces of $V_m$ is $m-1$ and it is denoted by $V_m^{-}$. We call this one non-split.
Define
\[
\e(V_m)=\begin{cases}+, \quad \text{if $V_m$ is split}\\
-, \quad \text{if $V_m$ is non-split}. \end{cases}
\]
Let $\O(V_m)$ be the associated orthogonal group of $V_m$. Observe that 
\[
\O(V_m)=\SO(V_m) \times \{\pm1\},
\]
where $\SO(V_m)$ is the special orthogonal group. 
Then $\SO(V_m)$ is split if and only if $V_m$ is split.
Given any irreducible admissible representation $\pi$ of $\SO(V_m)$, there are exactly two extensions of $\pi$ to $\O(V_m)$. It is determined by the action of $-1 \in \O(V_m)$; it acts on $V_m$ as a scalar multiple $1$ or $-1$. Therefore, the study of $\Irr(\O(V_m))$ is essentially that of $\Irr(\SO(V_m))$.

Let $\{e_1,\cdots,e_r\}$  and $\{e_1^*,\cdots,e_r^*\}$ be maximal isotropic subspaces of $V_m$ satisfying
\[
( e_i,e_j )_{V_m}=( e_i^*,e_j^* )_{V_m}=0, \quad ( e_i,e_j^* )_{V_m}=\delta_{ij}.
\]
For $1\le k \le r$, let 
\[
Y_k=\text{Span}\{e_1,\cdots,e_k\} \text{ and } Y_k^*=\text{Span}\{e_1^*,\cdots,e_k^*\},
\]
and $V_{m,k}$ be a subspace of $V_m$ 
so that $V_m=Y_k \oplus V_{m,k}\oplus Y_k^*$. If $V_m$ is split, let $e$ be an element in $V_{m,m}$ such that $(e,e)_{V_m}=1$.

Next, we consider the flag of isotropic subspaces
\[
Y_{k_1} \subseteq Y_{k_1+k_2} \subseteq \cdots \subseteq Y_{k_1+\cdots +k_r}\subseteq V_m.
\]
The stabilizer of such a flag is a parabolic subgroup $\P$ of $\O(V_m)$ whose Levi factor $\M$ is
\[
\M \simeq \GL_{k_1} \times \cdots \times \GL_{k_r} \times \O(V_{m,k_1+\cdots + k_r}),
\]
where each $\GL_{k_i}$ is the group of invertible linear maps on Span $\{e_{\sum_{t=1}^{i-1}k_t+1},\cdots,e_{\sum_{t=1}^{i}k_t}\}$.\\

 Throughout the paper, we write $\J_n$ for $\Sp(W_n)$ and $\wt{\J_n}$ for $\Mp(W_n)$. We also write $\H_m^{\pm}$ for $\SO(V_m^{\pm})$.

\subsubsection{Genuine representations}
Let $\psi$ be a non-trivial additive character of $F$, and let $\gamma_{\psi}$ be the Weil factor associated with $\psi$ (see the appendix of \cite{Rao93}). %The character $\chi_{\psi}$ of $\wt{\GL_k}(F)$ is defined as follows:
%\[
%\chi_{\psi}(g,\epsilon) = \epsilon \cdot \gamma_{\psi}(\det(g))^{-1},
%\]
%where $(g,\epsilon) \in \GL_k(F) \times \{\pm1\}$.

For a representation $\tau$ of $\GL_k(F)$, the genuine representation $\tau_{\psi}$ of $\wt{\GL_k}(F)$ is defined by
\[
\tau_{\psi}(a,\epsilon) = \epsilon \gamma_{\psi}(\det(a)) \tau(a), \quad (a,\epsilon) \in \GL_k(F) \times \{\pm1\}.
\]
Every irreducible admissible genuine representation of $\wt{\GL_k}(F)$ corresponds to a representation of $\GL_k(F)$ in this way. On the other hand, an irreducible representation $\wt{\pi}$ of $\wt{\J_n}(F)$ is called a genuine representation if $-1 \in \wt{\J_n}(F)$ does not act trivially.

Throughout this paper, when discussing irreducible representations of $\wt{\J_n}(F)$ or $\wt{\GL_k}(F)$, we implicitly assume that they are genuine representations.

\subsection{Generic characters and generic representations}\label{gen}
Let $\U$ be the unipotent radical of a Borel subgroup $\B=\T\U$ of $\H_m^+$, where $\T$ is the $F$-rational torus stabilizing the lines $Fv_i$ for each $i=1,\cdots,m$. 
Let $\mu$ be a generic character of $\U(F)$ defined by
\[\mu(u):=\psi((ue_2,e_1^*)_{V_m}+\cdots + (ue_m,e_{m-1}^*)_{V_m}+(ue,e_m^*)_{V_m}),\quad u \in \U(F).
\]
By definition, a representation $\pi$ of $\H_m^+$ is generic if $\Hom_{\U}(\pi,\mu) \ne 0$. Similarly, let $\U'$ be the unipotent radical of the Borel subgroup $\wt{\B}'=\wt{\T}'\U'$ of $\wt{\J_n}$, where $\T'$ is the $F$-split torus stabilizing the lines $Fw_i$ for each $i=1,\cdots,n$. 
For each $\lambda \in F^{\times}$, define a character $\mu_{\lambda}'$ of $\U'$ by
\[\mu_{\lambda}'(u'):=\psi(\pair{u'f_2,f_1^*}_{W_n}+\cdots +  \pair{u'f_{n},f_{n-1}^*}_{W_n}+\frac{\lambda}{2}\cdot\pair{u'f_n^*,f_n^*}_{W_n}), \quad u' \in \U'(F).\]

There is an action of $\T(F)$ (resp. $\T'(F)$) on the set of generic characters of $\U(F)$ (resp. $\U'(F)$), where the action is given by conjugation. It is easily checked that there is a unique $\T(F)$-orbit of generic characters of $\U(F)$, and thus, the definition of a generic representation of $\H_m^{+}(F)$ does not depend on the choice of generic character $\mu$. However, in the metaplectic cases, the $\T'(F)$-orbits of generic characters of $\U'(F)$ are indexed by non-trivial characters of $F^{\times}$ modulo the action of $F^{\times 2}$. More precisely, the map $\lambda \rightarrow \mu_{\lambda}'$ gives a bijection (depending on $\psi$):
\[
F^\times/F^{\times2}\rightarrow \text{\{$\T'(F)$-orbits of generic characters of $\U'(F)$\}}
\]
(see \cite[Sect. 12]{GGP}). Therefore, if $\Hom_{\U'}(\wt{\pi},\mu_{\lambda}') \ne 0$, we say that $\wt{\pi}$ is $\mu_{\lambda}'$-generic. Though the notion of genericity depends on $\psi$, when the choice of $\psi$ is clear in the context, we omit $\psi$.

We write $\Irr_{\text{gen}}(\wt{\J_n})$ for the subset of $\Irr(\wt{\J_n})$ consisting of $\mu_{1}'$-generic representations. Often, we call an element in $\Irr_{\text{gen}}(\wt{\J_n})$ a generic representation of $\wt{\J_n}(F)$. The notion of generic characters and generic automorphic representations can be similarly defined for global fields.

\subsection{$\gamma$-factors}
Let $\G$ be either $\H_m^{+}=\SO(V_m^{+})$ or $\wt{\J_n}=\Mp(W_n)$
. Let $\pi$ be an irreducible generic representation of $\G(F)$ and $\sigma$ an irreducible generic representation of $\GL_l(F)$. The local $\gamma$-factors can be defined in two distinct ways. The first approach is through the Langlands--Shahidi method (\cite{Sha90}, \cite{Sz10}, \cite{L15}), while the second approach relies on the proportionality between the functional equations of local Rankin--Selberg type integrals (or Shimura type integrals) for $\H_{m}^{+} \times \GL_l$ (or $\wt{\J_{n}} \times \GL_l$). It is a product of a monomial of $t=q^{-s}$ and a rational function $\frac{Q(t)}{P(t)}$, where $Q(t),P(t) \in \CC[t]$ such that $Q(0)=P(0)=1$. For a precise definition, please refer to \cite{So93} and \cite{Kap15}. It is worth noting that the two definitions of $\gamma$-factors are reconciled by several properties, as explained in \cite{Kap15}. Since we consider both the cases of $\operatorname{char}(F) = 0$ and $\operatorname{char}(F) = p$ simultaneously, we adopt the definition of $\gamma$-factors from the Langlands--Shahidi method, as they are defined in both cases. It is worthy to mention that Langlands-Shahidi method produces not only $\gamma$-factors but also local $L$-factors as follows:

Write $\gamma(s,\pi \times \sigma,\psi)=\alpha\cdot t^k \cdot \frac{Q(t)}{P(t)}$ for some $\alpha \in \CC^{\times}$ and polynomials $Q(t),P(t) \in \CC[t]$ such that  $\text{gcd}\big(Q(t),P(t)\big)=1$ and $Q(0)=P(0)=1$. Then $L_{\psi}(s,\pi \times \sigma,\psi)$ is defined by 
\begin{flalign}
\label{L-ftn}&L_{\psi}(s,\pi \times \sigma)\coloneqq Q(t)^{-1}.
\end{flalign}
Here the dependence on $\psi$ occurs only when $\pi$ is a representation of $\widetilde{\J_n}$. 

We next collect some fundamental properties of local $\gamma$-factors that will be needed later. While these properties are unconditional for $\G=\H_m^+$ (all characteristics) and for $\G=\wt{\J_n}$ in characteristic zero, the case of $\G=\wt{\J_n}$ in positive characteristic has not yet been fully developed. To proceed uniformly, we introduce the following working hypothesis.

\begin{hyp1}\label{hyp1}
The $\gamma$-factors for generic representations of $\wt{\J_n} \times \GL_l$ are properly defined in the case $\mathrm{char}(F) = p$. Furthermore, they satisfy properties (i)–(v) in Proposition~\ref{gamma}.
\end{hyp1}

We would like to recall some important properties of the $\gamma$-factors that will be utilized later in our discussion. (Caution : the property (vi) does not exist in the $\cha(F)=p$ case.)

\begin{prop}\label{gamma}
Let $\pi$ and $\sigma$ be irreducible admissible generic representations of $\G(F)$ and $\GL_l(F)$, respectively. 
Assume Working Hypothesis~\ref{hyp1} when $\G=\wt{\J_n}$ and $\operatorname{char}(F)=p$. 
Then the following properties hold:
\begin{enumerate}
\item (Unramified twist) $\gamma(s,\pi \times \sigma|\det|_F^{s_0},\psi)=\gamma(s+s_0,\pi \times \sigma,\psi)$ $\quad \text{ for } s_0 \in \RR.$
\item (Multiplicativity) Let $\tau_1 \otimes \cdots \otimes \tau_r \otimes \pi_0$ (resp. $(\tau_1)_{\psi} \otimes \cdots \otimes (\tau_r)_{\psi} \otimes \wt{\pi_0}$) be an irreducible admissible generic representation of $\GL_{r_1}(F)\times \cdots \times \GL_{r_r}(F) \times \H_{k}^{+}(F)$ (resp. $\wt{\GL_{r_1}}(F)\times_{\{\pm1\}} \cdots \times_{\{\pm1\}} \wt{\GL_{r_r}} (F) \times_{\{\pm1\}}\wt{\J}_{k}(F)$). Suppose that $\pi$ is the irreducible admissible generic constituent of a parabolically induced representation from $\tau_1 \otimes \cdots \otimes \tau_r \otimes \pi_0$ or $(\tau_1)_{\psi} \otimes \cdots \otimes (\tau_r)_{\psi} \otimes \wt{\pi_0}$. Then, 
\[\gamma(s,\pi \times \sigma,\psi)=\gamma(s,\pi_0 \times \sigma,\psi)\prod_{i=1}^r \gamma(s,\tau_i \times \sigma,\psi)\gamma(s,\tau_i^{\vee} \times \sigma,\psi).\]
Here, $\gamma(s,\tau_i \times \sigma,\psi)$ is the $\gamma$-factor for $\GL_{r_i} \times \GL_k$ of Jacquet, Piatetski-Shapiro and Shalika \cite{JSS}.
\item (Dependence on $\psi$) Given $a\in F^{\times}$, denote by $\psi_{a}$ the character of $F$ given by $\psi_{a}(x) \coloneqq \psi(a x)$ for $x \in F$. Let $\omega_{\pi}, \omega_{\sigma}$ be the central character of $\pi$ and $\sigma$, respectively. Then,
\[\gamma(s,\pi \times \sigma,\psi_a)= \omega_{\sigma}(a)^l \cdot |a|_F^{hl(s-\frac{1}{2})}\cdot  \gamma(s,\pi \times \sigma,\psi).
\]
Here, $h=2m$ if $\G=\H_m^+$; $h=2n$ if $\G=\wt{\J_n}$.
\item (Unramified factors) When all data are unramified, we have \[\gamma(s,\pi \times \sigma,\psi)=\frac{L_{\psi}(1-s,\pi^{\vee}\times \sigma^{\vee})}{L_{\psi}(s,\pi \times \sigma)}.
\]

\item (Global property: Functional equation)
Let $K$ be a global field with a ring of ad\`{e}les $\AA$ and $\Psi$ be a nontrivial character of $K \bs \AA$. Assume that $\Pi$ and $\Sigma$ are generic (depending on $\Psi$) cuspidal representations of $\G(\A)$ and $\GL_l(\AA)$, respectively. Let $S$ be a finite set of places of $F$ such that for $v \notin S$, all data are unramified. Then 
\[L_{\Psi}^S(s,\Pi \times \Sigma)=\prod_{v \in S} \gamma(s,\Pi_v \times \Sigma_v) \cdot L_{\Psi}^S(1-s,\Pi^{\vee} \times \Sigma^{\vee}).
\]
Here, $L_{\Psi}^S(s,\Pi \times \Sigma):= \displaystyle\prod_{v \notin S} L_{\Psi_v}(s, \Pi_v \times \Sigma_v)$ is the partial $L$-function with respect to $S$.
\item (Archimedean property)
For an archimedean field $F$, 
\[\gamma(s,\pi \times \sigma,\psi)=\gamma^{\textrm{Artin}}(s,\pi \times \sigma,\psi).\]
Here, $\gamma^{\textrm{Artin}}(s,\pi \times \sigma,\psi)$ is the Artin $\gamma$-factor under the local Langlands correspondence. 

\item (Tempered $L$-function)
Let $\pi,\sigma$ be irreducible tempered representations of $\G(F)$ and $\GL_k(F)$. Then $L_{\psi}(s,\pi \times \rho)$ is holomorphic for $\Re(s)>0$.
\end{enumerate}
\end{prop}

When $n=0$ and $m=0$, the $\gamma$-factors of $\H_0^{+} \times \GL_k$ and $\wt{\J_0}\times \GL_k$ are defined as follows:\\
Put $\mathbb{I}_{V_0^{+}}$ (resp. $\mathbb{I}_{W_0}$) the trivial representation of $\H_0^{+}(F)$ (resp. $\wt{\J_0}(F)$). Then for any irreducible generic representation $\sigma$ in $\text{Irr}(\GL_k)$,
\begin{flalign*} &\gamma(s,\mathbb{I}_{V_0^{+}} \times \sigma,\psi)\coloneqq 1 \\ &\gamma(s,\mathbb{I}_{W_0} \times \sigma,\psi)\coloneqq 1.
\end{flalign*}
It is clear that these definitions are compatible with the local functorial lifting.

In the case $\mathrm{char}(F)=0$, properties (i)--(vi) are established in 
\cite[Theorem~1]{Kap15} (see also \cite[Theorem~3.5]{Sha90}), 
while property (vii) is known only for $\G=\H_m^{+}$ 
(\cite[page 573]{CS}, \cite[Theorem~1.1]{HO13}).

In the case $\mathrm{char}(F)=p$, when $\G=\H_m^{+}$, 
properties (i)--(v) are proved in \cite[Theorem~1.4]{L15}, 
and property (vii) in \cite[Theorem~1.1]{CGL24}. 
For $\G=\wt{\J_n}$, however, the definition of generic $\gamma$-factors 
has not yet been established. 
Nevertheless, it is natural to expect that the methods of \cite{L15} 
extend to metaplectic groups as well. 
Accordingly, in this paper we adopt the above Working Hypothesis~\ref{hyp1} 
for $\G=\wt{\J_n}$ in positive characteristic.

Finally, to the best of the author’s knowledge, 
when $\mathrm{char}(F)=0$ the validity of property (vii) for $\G=\wt{\J_n}$ 
is still unknown. 
We therefore provide a proof in Corollary~\ref{chol}. 
Note that this proof relies on the Working Hypothesis~\ref{hyp1} 
in the case $\mathrm{char}(F)=p$, 
but is unconditional when $\mathrm{char}(F)=0$.

\begin{rem}\label{3gamma}
We make a brief remark about \( \gamma \)-factors in the case \( \operatorname{char}(F) = 0 \). Shahidi~\cite{Sha90} was the first to observe the fundamental properties of the Langlands--Shahidi \( \gamma \)-factors. When \( l = 1 \), Lapid--Rallis~\cite{LR05} and W.~T.~Gan~\cite{Gan2} defined \( \gamma \)-factors for irreducible smooth representations (including non-generic ones) of \( \G \times \GL_1 \), arising from the doubling method. They further established the so-called \emph{Ten Commandments}, which refer to ten fundamental properties that \( \gamma \)-factors satisfy.

Recently, Cai, Friedberg, and Kaplan (\cite{CFK22}) made a breakthrough by generalizing Lapid-Rallis’s Ten Commandments to $\H_m^+\times \GL_l$ for $l\ge 1$ arising from the doubling method to the twisted doubling method.
It is noteworthy that their $\gamma$-factors are defined not only for generic representations but also for non-generic representations of $\H_m^+(F) \times \GL_l(F)$ and they proved that three different definitions of $\gamma$-factors actually coincide for generic representations. \end{rem}

\begin{subsection}{Reduction of Theorem~\ref{main} to $\lambda=1$ cases}

In this subsection, our aim is to establish a reduction from Theorem~\ref{main} to Theorem~\ref{main1} below, which specifically focuses on $\mu_1'$-generic representations of $\wt{\J_n}(F)$.

In order to accomplish our goal, we rely on the following lemma.

\begin{lem}\label{prop:mu_d} A $\T'(F)$-orbit of irreducible $\mu_{\lambda}'$-generic representations with respect to $\psi$ is equals to a $\T'(F)$-orbit of irreducible admissible $\mu_{1}'$-generic representations with respect to $\psi_{\lambda}$.
\end{lem}
\begin{proof}
Let $\pi$ be an irreducible $\mu_{\lambda}'$-generic representation of $\wt{\J_n}(F)$ with respect to $\psi$. We aim to show that $\pi$ is $\mu_1'$-generic with respect to $\psi_{\lambda}$.

Consider the action of $t \in \T'(F)$ on a generic character $\chi'$ of $\U'(F)$, given by $(\chi')^t(u') = \chi'(t^{-1}u't)$. Specifically, for $t \in \T'(F)$ with $t(f_i) = t_if_i$ and $t(f_i^*) = t_i^{-1}f_i^*$ $(1 \leq i \leq n)$, we have
$$(\mu_{\lambda}')^t(u') = \psi\left(\sum_{i=1}^{n-1}t_i^{-1}t_{i+1}\langle u'f_i,f_{i+1}^*\rangle_{W_n} + \lambda t_n^{-2}\frac{\langle u'f_n^*,f_n^*\rangle_{W_n}}{2}\right).$$

In particular, if we choose $\underline{t} \in \T'(F)$ specifically with $t_i = \lambda^{-(n-i)}$, then we obtain
$$(\mu_{\lambda}')^{\underline{t}}(u') = \psi_{\lambda}\left(\sum_{i=1}^{-1}\langle u'f_i,f_{i+1}^*\rangle_{W_n} + \frac{\langle u'f_n^*,f_n^*\rangle_{W_n}}{2}\right).$$

Therefore, a $(\mu_{\lambda}')^{\underline{t}}$-generic representation (with respect to $\psi$) is $\mu_1'$-generic with respect to $\psi_{\lambda}$. Since the notion of genericity is invariant under $\T'(F)$-orbits, it follows that $\pi$ is $\mu_1'$-generic with respect to $\psi_{\lambda}$.

This completes the proof.
\end{proof}

The lemma above and Proposition~\ref{gamma} (iii) imply that it suffices to consider $\lambda=1$ cases to prove Theorem~\ref{main}. Hence, for the remainder of the paper, we will focus on establishing the following theorem, which serves as a replacement for Theorem~\ref{main}.

\begin{thm}[Local Converse Theorem for $\wt{\Sp}_{2n}$] \label{main1} Let $\wt{\pi}$ and $\wt{\pi}'$ be irreducible admissible $\mu_1'$-generic representations of $\wt{\J_n}(F)$ such that
\[\gamma(s,\wt{\pi} \times \sigma,\psi)=\gamma(s,\wt{\pi}' \times \sigma,\psi)
\] holds for all irreducible generic supercuspidal representation $\sigma$ of $\GL_i(F)$ with $1\le i \le n$. Then
\[\wt{\pi} \simeq \wt{\pi}'.
\]
\end{thm}
\end{subsection}
\newpage
\section{Local theta correspondence}

In this section, we introduce the local theta correspondence induced by a Weil representation $\omega_{\psi,V_m^{\pm},W_n}$ of $\H_m^{\pm}(F) \times \wt{\J}_n(F)$ and recapitulate some fundamental results associated with it. Subsequently, we prove two theorems that will play crucial roles in establishing our main theorem in the next section.

The group $\H_m^{\pm}(F) \times \wt{\J}_n(F)$ naturally possesses a distinguished representation known as the Weil representation, denoted as $\omega=\omega_{\psi,V_m^{\pm},W_n}$. For every $\pi\in\Irr(\H_m^{\pm})$, the maximal $\pi$-isotypic quotient of $\omega$ can be expressed as
\[
\pi\boxtimes \Theta_{\psi,W_n,V_m^{\pm}}(\pi),
\]
where $\Theta_{\psi,W_n,V_m^{\pm}}(\pi)$ denotes a smooth finite-length representation of $\wt{\J_n}(F)$ referred to as the big theta lift of $\pi$. The small theta lift of $\pi$ corresponds to the maximal semisimple quotient of $\Theta_{\psi,W_n,V_m^{\pm}}(\pi)$. Similarly, for each $\wt{\pi}\in\Irr(\wt{\J_n})$, we obtain a smooth finite-length representation $\Theta_{\psi,V_m^{\pm},W_n}(\wt{\pi})$ of $\H_m^{\pm}(F)$. The maximal semisimple quotient of $\Theta_{\psi,V_m^{\pm},W_n}(\wt{\pi})$ (and $\Theta_{\psi,W_n,V_m^{\pm}}(\pi)$) is denoted as $\theta_{\psi,V_m^{\pm},W_n}(\wt{\pi})$ (and $\theta_{\psi,W_n,V_m^{\pm}}(\pi)$) respectively. 

\begin{thm}[Howe duality] \label{howe}For $\wt{\pi}\in\Irr(\wt{\J_n})$, if  $\theta_{\psi,V_m^{\pm},W_n}(\wt{\pi})$ is nonzero, it is irreducible. Furthermore, for $\wt{\pi}_1,\wt{\pi}_2\in\Irr(\wt{\J_n})$, if $\theta_{\psi,V_m^{\pm},W_n}(\wt{\pi}_1)$ and $\theta_{\psi,V_m^{\pm},W_n}(\wt{\pi}_2)$ are nonzero and $\theta_{\psi,V_m^{\pm},W_n}(\wt{\pi}_1) \simeq \theta_{\psi,V_m^{\pm},W_n}(\wt{\pi}_2)$, then $\wt{\pi}_1 \simeq \wt{\pi}_2$. The analogous statements hold for $\pi,\pi_1,\pi_2\in\Irr(\H_m^{\pm})$. 
\end{thm}

The Howe duality was originally proved by Waldspurger~\cite{Wa90} for local fields of characteristic zero with residual characteristic not equal to~$2$, and later extended to a more general setting by Gan--Takeda~\cite{GT1,GT2}. At archimedean places, the conjecture was proved by Howe himself.

For the rest of this paper, we conveniently denote $V_n$ and $W_n$ as $V$ and $W$, respectively. Given $\wt{\pi} \in \Irr(\wt{\J_n})$, there exists a unique $\e \in \{\pm1\}$ such that $\theta_{\psi,V^{\e},W}(\wt{\pi})$ is nonzero (see \cite[Theorem 1.4]{GS}).

Using the theta correspondence between $\wt{\J_n}$ and $\H_n^{\pm}$, we define a mapping $\Theta_{\psi}:\Irr(\wt{\J_n}) \longrightarrow \Irr(\H_n^{+}) \sqcup \Irr(\H_n^{-})$ for each non-trivial additive character $\psi:F \longrightarrow \CC^{\times}$. Specifically, for $\wt{\pi} \in \Irr(\wt{\J_n})$, we have $\Theta_{\psi}(\wt{\pi})=\theta_{\psi,V^{\e},W}(\wt{\pi})$, where $\e \in \{\pm1\}$ is chosen uniquely such that $\theta_{\psi,V^{\e},W}(\wt{\pi})$ is nonzero. The inverse of this map is denoted by $\Theta_{\psi}^{-1}(\pi)=\theta_{\psi,W,V^{\pm}}(\pi)$, which is nonzero for any $\pi \in \text{Irr}(\H_n^{\pm})$.

The bijection $\Theta_{\psi}$ preserves various properties. Here are some of them that will be useful later.

%One of the main results in \cite{GS} is the following.
%\begin{thm}\cite[Theorem 1.1]{GS} Fix a non-trivial additive character $\psi:F \raw \CC^{\times}$, Then the map $\Theta_{\psi}$
%\[\Theta_{\psi}:\Irr(\wt{\J_n}) \longrightarrow \Irr(\G_n^{+}) \sqcup \Irr(\G_n^{-}).
%\]
%is bijective.
%\end{thm}

\begin{thm}[{\cite[Theorem C.4]{GI1}}]\label{tem0} Let $\pi \in \Irr_{\temp}(\H_n^{\pm})$ and  $\wt{\pi} \in \Irr_{\temp}(\wt{\J_n})$. Then $\Theta_{\psi,W,V^{\pm}}(\pi) \in \Irr_{\temp}(\wt{\J_n})$. Similarly, if $\Theta_{\psi,V^{\pm},W}(\wt{\pi})$ is non-zero, then $\Theta_{\psi,V^{\pm},W}(\wt{\pi}) \in \Irr_{\temp}(\H_n^{\pm})$. 
\end{thm}
\begin{thm}[{\cite[Theorem 2.2]{JS03}}]\label{sup} Let $\pi \in \Irr_{\supc}(\H_n^{+})$ and  $\wt{\pi} \in \Irr_{\supc}(\wt{\J_n})$. Suppose that $\pi$ is generic and $\wt{\pi}$ is $\mu_1'$-generic. Then $\Theta_{\psi,W,V^+}(\pi)$ is supercuspidal and \[\Theta_{\psi,V^+,W}(\wt{\pi})=\begin{cases} \rm{Steinberg \ representation}, \ \rm{  if \ }  n=1  \rm{ \ and \ } \wt{\pi} \rm{ \ is \  odd \ Weil\ representation\ } \\ \rm{supercuspidal}, \quad \quad \quad \quad \ \ \rm{\ otherwise. }   \end{cases}\]
\end{thm}

%The following theorem extends a part of \cite[Corollary 2.1 (ii)]{JS03} from supercuspidal representations to tempered representations.
\begin{thm}\label{Theta1}
Let $\wt{\pi} \in \Irr_{\temp}(\wt{\J_n})$. If $\wt{\pi}$ is $\mu_1'$-generic, then $\Theta_{\psi}(\wt{\pi})$ is a tempered generic representation of the split odd orthogonal group $\H_n^{+}(F)$.
\end{thm}
When $\mathrm{char}(F) = 0$, the theorem above can be seen as an extension of \cite[Corollary 2.1 (2)]{JS03}, expanding from supercuspidal to tempered representations. (In fact, \cite[Corollary 2.1 (2)]{JS03} provides even more.) Moreover, when the residual characteristic of $F$ is odd, this result is stated in \cite[Corollary~9.3 (ii)]{GS}, without proof. To demonstrate that their omitted proof also applies to the case $\mathrm{char}(F) = p$, as well as $\mathrm{char}(F) = 0$, we present a proof here. This proof is based on the following lemma, which calculates the twisted Jacquet module of the Weil representation as stated (but not proven) in \cite[Proposition~9.2]{GS}.

\begin{lem}[{\cite[Proposition~{9.2}]{GS}}]\label{tjw}
Let $(\omega_{\psi,V^{+},W})_{\U,\mu^{-1}}$ be the twisted Jacquet module of $\omega_{\psi,V^{+},W}$ with respect to $\U$ and $\mu^{-1}$. Then the following holds:
\[(\omega_{\psi,V^{+},W})_{\U,\mu^{-1}} \cong \ind_{\U'}^{\wt{\J_{n}}} (\mu_1').\]
\end{lem}

The proof of Lemma~\ref{tjw} will be provided in the next section. We can establish Theorem~\ref{Theta1} assuming Lemma~\ref{tjw} as follows.

\begin{proof}[Proof of Theorem~\ref{Theta1}]
Write $(\omega_{\psi,V^{+},W})_{\U,\mu^{-1}}$ for the twisted Jacquet module of $\omega_{\psi,V^{+},W}$ with respect to $\U$ and $\mu^{-1}$. By Lemma~\ref{tjw}, we have
\begin{align*}
\Hom_{\wt{\J_n} \times \U}\big(\omega_{\psi,V^{+},W}, \ \wt{\pi} \otimes \mu^{-1}\big) &\cong \Hom_{\wt{\J_n}}\big((\omega_{\psi,V^{+},W})_{\U,\mu^{-1}}, \ \wt{\pi}\big)\\
&\cong \Hom_{\wt{\J_n}}\big(\ind_{\U'}^{\wt{\J_n}} (\mu_1'), \ \wt{\pi} \big)\\
&\cong \Hom_{\U'}\big(\mu_1', \ (\wt{\pi}^{\vee}\mid_{\U'})^{\vee} \big)\\
&\cong \Hom_{\U'}\big(\wt{\pi}^{\vee}, \ (\mu_1')^{\vee} \big)\\
&\cong \Hom_{\U'}\big(\wt{\pi}, \ \mu_1' \big),
\end{align*}
where the last equality follows from the facts that $\wt{\pi}$ and $\mu_1'$ are unitary.

On the other hand,
\[
\Hom_{\wt{\J_n} \times \U}\big(\omega_{\psi,V^{+},W}, \ \wt{\pi} \otimes \mu^{-1}\big)\cong \Hom_{\U}\big(\Theta_{\psi,V^{+},W}(\wt{\pi}), \  \mu^{-1}\big).
\]

If $\wt{\pi}'$ is $\mu_1'$-generic, then $\Hom_{\U'}\big(\wt{\pi}', \ \mu_1' \big) \ne 0$, and thus $\Hom_{\U}\big(\Theta_{\psi,V^{+},W}(\wt{\pi}'), \  \mu^{-1}\big) \ne 0$. This, combined with Theorem~\ref{tem0}, proves the theorem and $\Theta_{\psi,V^{+},W}(\wt{\pi}')$ is a representation of $\H_n^{+}$ because it is generic.
\end{proof}

\begin{rem}
The opposite direction of Theorem~\ref{Theta1}, that is starting with a tempered generic representation of $\H_n^{+}(F)$ and its theta lift to $\wt{\J_n}(F)$ is generic, is proved in greater generality in \cite[Corollary~{2.1 (1)}]{JS03}. 
Note that when $n=1$, even though $\wt{\pi}$ is $\mu_1'$-generic and a supercuspidal representation, $\Theta_{\psi}(\wt{\pi}')$ might not be supercuspidal. For example, if $\wt{\pi}$ is the odd Weil representation, then $\Theta_{\psi}(\wt{\pi}')$ is a Steinberg representation. For this reason, Jiang and Soudry separated the proof of the LCT for $\H_n^{+}$ according to whether $n=1$ or $n \ge 2$. In this paper, to ensure our arguments also work for $n=1$ cases, we reduce the proof of our main theorem to the tempered cases, not the supercupidal cases. This contrasts with the approach taken in \cite{JS03}, where they reduced it to the supercuspidal cases. This is one of the delicate parts in this paper.
\end{rem}

Consider an induced representation 
\[
\Ind_{\cl{Q}}^{\Mp(W)}
((\tau_1|\det|_F^{s_1})_{\psi}\otimes\cdots\otimes(\tau_r|\det|_F^{s_r})_{\psi}\otimes \cl{\pi}_0),
\]
where the following conditions hold:
\begin{itemize}
\item $W=W_{n}$ and $W_{n_0}$ are symplectic spaces of dimension $2n$ and $2n_0$, respectively.
\item $\cl{Q}$ is a parabolic subgroup of $\Mp(W)$ with the Levi subgroup isomorphic to $\cl{\GL}_{n_1}\times_{\{\pm1\}} \cdots\times_{\{\pm1\}}\cl{\GL}_{n_r}\times_{\{\pm1\}} \Mp(W_0)$.
\item $\tau_i$ is an irreducible unitary supercuspidal representation of $\GL_{n_i}(F)$.
\item $s_i$ is a real number such that $s_1 \geq \dots \geq s_r > 0$.
\item $\cl{\pi}_0$ is an irreducible $\mu_1'$-generic supercuspidal (resp. tempered) representation of $\Mp(W_0)(F)$.
\end{itemize}

An irreducible \( \mu_1' \)-generic representation of \( \wt{\J_n}(F) \) is a subquotient of such an induced representation (cf.~\cite[p.~777]{JS03}). In this case, we say that it has supercuspidal (resp. tempered) support \( (\tau_1, \ldots, \tau_r; \wt{\pi}_0) \) and exponents \( (s_1, \ldots, s_r) \). Sometimes, we omit mentioning the exponents, in which case we simply say that 
\[
(\tau_1 \cdot |\cdot|_F^{s_1}, \ldots, \tau_r \cdot |\cdot|_F^{s_r}; \wt{\pi}_0)
\]
is a supercuspidal (resp.~tempered) support of the given representation. The supercuspidal (resp.~tempered) support and the exponents of an irreducible \( \mu_1' \)-generic representation of \( \wt{\J_n}(F) \) are uniquely determined up to permutation (see Corollary~\ref{uniq}). The notions of supercuspidal (resp.~tempered) support and exponents for an irreducible generic representation of \( \H_m^{+}(F) \) are defined similarly.

The following is a special case of Kudla's supercuspidal support theorem.

\begin{thm} [{\cite[Theorem 7.1, Theorem 7.2]{Ku} or \cite[Proposition 5.2]{GI1}}]\label{ksc} 
Let $\wt{\pi} \in \Irr(\wt{\J_n})$ be a $\mu_1'$-generic representation. Suppose that $\Theta_{\psi,V_m^{+},W}(\wt{\pi})$ is nonzero. Put $\pi=\Theta_{\psi,V_m^{+},W}(\wt{\pi})$. Assume that $\wt{\pi} \in \Irr(\wt{\J_n})$  has supercuspidal support $(\wt{\tau_1},\cdots,\wt{\tau_n};\wt{\pi_0})$, where $\wt{\pi_0} \in \Irr(\wt{\J_{0}})$ is the trivial representation and $\tau_i$'s are characters of $\GL_1(F)$. Write the supercuspidal support of $\pi \in \Irr(\H_m^{+})$ as $(\tau_1,\cdots,\tau_r;\pi_0)$, where $\pi_0 \in \Irr(\H_{m_0}^{+})$. Then for each $m=n,n-1$, the following holds:
\begin{enumerate}
\item $r=m$ and $m_0=0$.
\item $\pi_0$ is the trivial representation
\item $\{ \wt{\tau_1},\cdots,\wt{\tau_n}\}= \begin{cases}\{\tau_1,\cdots,\tau_n \},\quad \quad \quad \quad  \ \rm{if}\  m=n  \\
\{|\cdot|_F^{\frac{1}{2}}
,\tau_1,\cdots,\tau_{n-1}\},\quad  \ \rm{if} \  m=n-1.
\end{cases}$
%\item $\{ s_1,\cdots,s_n\}= \begin{cases}\{s_1',\cdots,s_n' \},\quad \quad \quad \quad \ \rm{if}\  m=n  \\
%\{\frac{1}{2}
%,s_1',\cdots,s_{n-1}'\},\quad \quad \ \rm{if} \  m=n-1.
%\end{cases}$
\end{enumerate}
\end{thm}
By applying Kudla's supercuspidal support theorem, we can establish that $\Theta_{\psi}$ preserves $\gamma$-factors.

\begin{thm} \label{Theta} For $\wt{\pi} \in \Irr_{\gen}(\wt{\J_n})$, suppose that $\pi=\Theta_{\psi}(\wt{\pi})$ is generic. Then for any irreducible generic representation $\sigma$ of $\GL_r(F)$,
    \[\gamma(s,\wt{\pi} \times \sigma, \psi)=\gamma(s,\pi \times \sigma,\psi).    
    \]
\end{thm}
\begin{rem}\label{err}
   When $\cha(F)=0$, this theorem is stated in \cite[Proposition~{11.1}]{GS} and the discussion after \cite[Corollary~{11.2}]{GS}. In the proof provided by \cite[Proposition~11.1]{GS}, they construct first $\Pi$ (the globalization of $\pi$) and consider its first occurrence, denoted as $\Theta_{k_0}(\Pi)$, in the global theta lifting tower from $\H_n^{+}(\A)$ to the tower of $\wt{\J_{2k}}(\mathbb{A})$. They then compare the $\gamma$-factors of $\Pi_v$ and $\Theta_{k_0}(\Pi)_v$ at unramified places. However, the genericity of $\Theta_{k_0}(\Pi)_v$ remains uncertain, as there is no guarantee that $\Theta_{k_0}(\Pi)$ is globally generic. Since local $\gamma$-factors of $\wt{\J_{k_0}}(F_v)$ are only defined for generic representations, we cannot directly compare the local $\gamma$-factors of $\Pi_v$ and $\Theta_{k_0}(\Pi)_v$ if $\Theta_{k_0}(\Pi)_v$ happens to be non-generic. Therefore, we provide a detailed proof circumventing this issue. 
\end{rem}

\begin{proof}[Proof of Theorem~\ref{Theta}]
We begin by proving the case when $\wt{\pi}$ and $\sigma$ have  nonzero Iwahori-fixed vectors. In this case, we can express $\sigma$ as a submodule of a principal series representation of $\GL_r(F)$. By utilizing the multiplicative property of $\gamma$-factors, it suffices to consider the case when $\sigma$ is a $1$-dimensional unramified character of $F^{\times}$. 

Assume that $(\wt{\rho_1},\cdots,\wt{\rho_r};\wt{\pi}_0)$  and $(\rho_1,\cdots,\rho_s;\pi_0)$ are the supercuspidal supports of $\wt{\pi}$ and $\pi$, respectively. 
Since $V_{0}^{+}$ (resp. $W_{0}$) is the anisotropic kernel of $V^{+}$ (resp. $W$), $\wt{\pi}_0=\mathbb{I}_{V_{0}^{+}}$ (resp. $\pi_0=\mathbb{I}_{W_{0}}$) is the trivial representation of $\H_{0}^{+}(F)$ (resp. $\wt{\J_{0}}(F)$). Furthermore, $\wt{\rho_i},\rho_j$ are 1-dimensional unramified character of $\GL_1(F)$.
Then by Proposition~\ref{gamma} (ii) and Theorem~\ref{ksc},
\begin{align*}\gamma(s,\pi \times \sigma,\psi)&=\gamma(s,\mathbb{I}_{V_{0}^{+}} \times \sigma,\psi)\cdot \prod_{i=1}^s \gamma(s,\rho_i \times \sigma,\psi)\cdot \gamma(s,\rho_i^{\vee} \times \sigma,\psi) \\ &=\gamma(s,\mathbb{I}_{W_{0}} \times \sigma,\psi)\cdot \prod_{i=1}^r \gamma(s,\wt{\rho_i} \times \sigma,\psi)\cdot \gamma(s,(\wt{\rho_i})^{\vee} \times \sigma,\psi) \\ &=\gamma(s,\wt{\pi} \times \sigma,\psi). \end{align*}
Next, we consider the general case. Again, using the multiplicative property of $\gamma$-factors, we only need to consider the case when both $\wt{\pi}$ and $\sigma$ are supercuspidal. At this point, we employ a global-to-local argument. Let us introduce the following data:

\begin{itemize}
\item $K$: a totally imaginary number field (or a global function field) with $\A$ as its ad\`{e}le ring such that $K_{v_0}=F$ for a specific finite place $v_0$ of $K$.
 \item $\Psi$: a nontrivial additive character of $K \bs \A$ satisfying $\Psi_{v_0}=\psi$.
\item $\mf{U}'$: a nontrivial additive generic character of $\U'(K) \bs \U'(\A)$ associated with $\Psi$, such that $\mf{U}_{v_0}'=\mu_{1}'$.
\item $\mathbb{\WW}$: a symplectic space over $K$ of dimension $2n$, with the associated isometry group $\Sp(\WW)$.
\item $\mathbb{V}$: a split symmetric space over $K$ of dimension $2n+1$, such that $\mathbb{V}_{v_0}=V^+$, with the associated isometry group $\SO(\VV)$.
\item $\{\VV_k\}_{k \ge 0}$: the Witt tower over $K$ containing $\VV_n=\VV$.
\item $\Sigma$: a globally generic cuspidal representation of $\GL_r(\A)$, such that $\Sigma_{v_0}=\sigma$, and for all finite places $v \neq v_0$ of $K$, the local components $\Sigma_{v}$ are unramified.
\end{itemize}

The existence of such $K$ follows from a similar argument in \cite[Lemma~{5.2}]{MS20}. The construction of such $\Sigma$ is indeed possible, as stated in \cite[Theorem 4.1]{PS} for the case when $\cha(F)=0$ and in \cite[Theorem 1.1]{GL} for the case when $\cha(F)=p$.

To prove the desired result, we consider the cases where $\pi = \Theta_{\psi}(\wt{\pi})$ is either a supercuspidal representation or a Steinberg representation. We start by addressing the case in which $\pi$ is supercuspidal.

Construct a globally $\mf{U}'(\AA)$-generic cuspidal representation $\Pi$ of $\Mp(\WW)(\AA)$, adapting the constructions from \cite[Theorem 4.1]{PS} and \cite[Theorem 1.1]{GL} to metaplectic groups. This representation $\Pi$ satisfies $\Pi_{v_0}=\wt{\pi}$ and for all finite places $v \ne v_0$ of $K$, $\Pi_v$ is unramified.

Let $\Xi$ be the global theta lift of $\Pi$ to $\SO(\VV)(\A)$. Using \cite[Proposition 3]{Fu}, we know that $\Xi$ is nonzero and generic. We claim that $\Xi$ is cuspidal. Suppose not, which means that there exists an integer $m<n$ and a $(2m+1)$-dimensional symmetric space $\VV_m$ over $K$ such that the global theta lift $\Theta_m(\Pi)$ of $\Pi$ to $\SO(\VV_m)(\A)$ is cuspidal. Let $\Theta_m(\Pi)_0$ be an irreducible summand of $\Theta_m(\Pi)$. Then $\Theta_m(\Pi)_{0,v_0}$ is the nonzero local theta lift of $\wt{\pi}$ to $\SO(\VV_{m})(F)$. However, this contradicts the tower property of local theta lifts (see \cite[Theorem~{6.1}]{Ku}) since $\pi$ is a supercuspidal representation of $\H_n^+(F)$. Therefore, we have shown that $\Xi$ is a cuspidal representation of $\SO(\VV)(\A)$ with $\Xi_{v_0}=\theta_{\psi,V,W}(\wt{\pi})$. 

Let $\Xi_0$ be an irreducible generic constituent of $\Xi$. (Indeed, $\Xi$ is irreducible by the multiplicity one property in the Howe duality theorem.) By our first claim, for all unramified places $v$ of $K$, we have
\[
\gamma(s,\Xi_{0,v} \times \Sigma_{v},\Psi_v)=\gamma(s,\Pi_v \times \Sigma_{v},\Psi_v).
\]
For any archimedean place $v$ of $K$, $K_v = \mathbb{C}$. The theta correspondence for complex groups is well-understood and can be described in terms of the local Langlands correspondence (see \cite{AB95}). Applying the archimedean property of $\gamma$-factors, we have the above identity at the archimedean places of $K$ as well.

Put $S_0$ the set of all archimedean places of $K$ and $S=S_0 \cup \{v_0\}$ (In $\cha(F)=p$ cases, put $S=\{v_0\}$). Note that for all places $v \notin S$, $\Xi_{0,v}, \Pi_v$ and $\Sigma_v$ are unramified, and for all archimedean places $v$, $K_v=\mathbb{C}$. By the global functional equation of $\gamma$-factors,
\begin{align*}
1=&\,\prod_{v\in S}\gamma(s,\Xi_{0,v} \times \Sigma_{v},\Psi_{v}) \cdot \frac{L^S(1-s,\Xi_0^{\vee} \times \Sigma^{\vee})}{L^S(s,\Xi_0 \times \Sigma)} \\
=&\,\prod_{v\in S}\gamma(s,\Pi_{v}\times \Sigma_{v},\Psi_{v})\cdot \frac{L^S(1-s,\Pi^{\vee} \times \Sigma^{\vee})}{L^S(s,\Pi \times \Sigma)}.
\end{align*}
By canceling the identities at places outside $v_0$, we have
\[
\gamma(s,\Xi_{0,v_0} \times \Sigma_{v_0},\Psi_{v_0}) = \gamma(s,\Pi_{v_0}\times \Sigma_{v_0},\Psi_{v_0}).
\]
Since $\Pi_{v_0}=\wt{\pi}'$, $\Sigma_{v_0}=\sigma$, and $\Xi_{0,v_0}=\pi$, we get the desired identity.

Now, let's consider the case when $\pi$ is the Steinberg representation. We choose another finite place $v_1 \ne v_0$ of $K$ and a supercuspidal representation $\pi_1$ of $\H_n^+(K_{v_1})$. Using similar constructions as in \cite[Theorem 4.1]{PS} and \cite[Theorem 1.1]{GL}, we can construct a $\mf{U}'(\A)$-generic cuspidal representation $\Pi'$ of $\Mp(\WW)(\AA)$ such that $\Pi'_{v_0} = \wt{\pi}$, $\Pi'_{v_1} = \theta_{\psi,\WW_{v_1},\VV_{v_1}}(\pi_1)$, and it is unramified at other finite places $v \ne v_0, v_1$ of $K$.

Let $\Xi_0$ be an irreducible generic constituent of the global theta lift $\Theta(\Pi')$ of $\Pi'$ to $\SO(\VV)(\AA)$. Then, $\Xi_{0,v_0} \simeq \pi$, $\Xi_{0,v_1} \simeq \pi_1$, and so $\Xi_0$ is cuspidal because $\pi_1$ is supercuspidal.

We have already shown that
\[
\gamma(s,\Xi_{0,v_1} \times \Sigma_{v_1},\Psi_{v_1}) 
=\gamma(s,\Pi_{v_1}'\times \Sigma_{v_1},\Psi_{v_1})
\]
since $\Xi_{0,v_1} \simeq \pi_1$ is supercuspidal.

For other places $v \ne v_0, v_1$, we know that
\[
\gamma(s,\Xi_{0,v} \times \Sigma_{v},\Psi_{v}) 
=\gamma(s,\Pi_{v}'\times \Sigma_{v},\Psi_{v})
\]from the previous argument.

Therefore, by canceling the above identities at all places other than $v = v_0$ in the two global functional equations, we obtain the desired identity.
\end{proof}
From Theorem~\ref{Theta1}, we can prove Proposition~\ref{gamma} (vii) for $\G =\wt{\J_n}$.

\begin{cor}\label{chol}Let $\wt{\pi}$ and $\sigma$ be irreducible generic tempered representations of $\wt{\J_n}(F)$ and $\GL_k(F)$, respectively. Then $L_{\psi}(s,\wt{\pi} \times \sigma)$ is holomorphic for $\Re(s)>0$.
\end{cor}
\begin{proof}Write $\pi = \Theta_{\psi}(\wt{\pi})$. By Theorem~\ref{Theta1}, $\pi$ is tempered and generic. Hence, by Theorem~\ref{Theta},
\[
\gamma(s, \wt{\pi} \times \sigma, \psi) = \gamma(s, \pi \times \sigma, \psi).
\]
From the definition of the local $L$-function (\ref{L-ftn}), we have
\[
L_{\psi}(s, \wt{\pi} \times \sigma) = L_{\psi}(s, \pi \times \sigma).
\]
By Proposition~\ref{gamma} (vii) for $\G = \H_n^{+}$, $L_{\psi}(s, \pi \times \sigma)$ is holomorphic for $\Re(s) > 0$, and therefore, so is $L_{\psi}(s, \wt{\pi} \times \sigma)$.
\end{proof}

\newpage
\begin{section}{Computation of the twisted Jacquet module of the Weil representation}
\noindent In this section, we establish Lemma~\ref{tjw}, which plays a crucial role in proving Theorem~\ref{Theta1}.

 Let $V=V_n^{+}$ and $W=W_n$, and consider a symplectic form $(\cdot,\cdot)$ defined on $V \otimes W$ as follows:
\[(v_1 \otimes w_1, v_2 \otimes w_2)=( v_1,v_2 )_{V}\cdot \la w_1,w_2 \ra_{W}.\]

There is a natural embedding of $\H_n^{+} \times \wt{\J_n}$ into $\wt{\Sp}(V \otimes W)$. By restricting the Weil representation $\omega_{\psi,V,W}$ of $\wt{\Sp}(V\otimes W)(F)$ to the image of $\H_n^{+}(F) \times \wt{\J_n}(F)$ inside $\wt{\Sp}(V \otimes W)(F)$, we regard $\omega_{\psi,V,W}$ as a representation of $\H_n^{+}(F) \times \wt{\J_n}(F)$.

Recall that $V=Y_n \oplus e \oplus Y_n^*$ and $W=X_n\oplus X_n^*$. There are two kinds of polarization for $V \otimes W$. One is $(Y_n^* \otimes W) \oplus (e \otimes X_n)$, and the other is $V \otimes X_n^*$. Depending on the chosen polarization, the description of the associated Weil representation will differ. For our purpose, we use the polarization $(Y_n^* \otimes W) \oplus (e \otimes X_n)$ of $V \otimes W$. In this case, there is an action of $\H_n^{+}(F) \times \wt{\J_n}(F)$ (resp. $\wt{\J_n}(F)$) on the Schwartz-Bruhat function space $\big(S(Y_n^* \otimes W) \otimes S(e \otimes X_n)\big)(F)$ (resp. $S(X_n)(F)$) via the mixed Schrödinger model of the Weil representation $\omega_{\psi,V,W}$ (resp. $\omega_{\psi,W}$).

Let $\P_n=\M_n \V_n$ be a parabolic subgroup of $\H_n^{+}$ that stabilizes $Y_n$, where $\M_n$ is the Levi subgroup and $\V_n$ is the maximal unipotent subgroup of $\M_n$. Write $\N_n= \{\alpha \in \Hom(Y_n^*,Y_n)\ \vert \ \alpha^*=-\alpha\},$
where $\alpha^*$ is the element in $\Hom(Y_n^*,Y_n)$ satisfying 
\[( ay_1,y_2)_V= (y_1,a^*y_2)_V, \quad \text{ for all } y_1,y_2 \in Y_n^*.\]
Then
\[\M_n\simeq \GL(Y_n) \text{ and } \V_n\simeq \Hom(e,Y_n) \ltimes \N_n.\]   
Let $\Z_n$ be the maximal unipotent subgroup of $\M_n$. The action of $\Z_n(F)\cdot \Hom(e,Y_n) \cdot \N_n(F)\times \wt{\J_n}(F)$ on $S(Y_n^* \otimes W)(F) \otimes S(e \otimes X_n)(F)$ via the mixed Schrödinger model can be described as follows (see \cite[\S{II}.7, Chapter 2]{MVW}). For $\phi_1\otimes \phi_2 \in S(Y_n^* \otimes W)(F) \otimes S(e \otimes X_n)(F)$ and $(w,x) \in (Y_n^* \otimes W)(F) \times (e \otimes X_n)(F)$,
\begin{itemize}
 \item $\omega_{\psi,V,W}(1,h)(\phi_1\otimes \phi_2)(w,x)=\phi_1(w\bar{h})\cdot (\omega_{\psi,W}(h)\phi_2)(x) \text{ for } h \in \wt{\J_n}(F),$
 \item $\omega_{\psi,V,W}\big(z,1\big)(\phi_1\otimes \phi_2)(w,x)=\phi_1\big(z^* w \big)\cdot \phi_2(x)  \text{ for } z \in \Z_n(F) \ss \GL(Y_n)(F),$
 \item $\omega_{\psi,V,W}(t,1)(\phi_1\otimes \phi_2)(w,x)=\psi\big((t^*w,x)\big)\cdot(\phi_1\otimes \phi_2)(w,x)  \text{ for } t\in \Hom(e,Y_n)(F),$
 \item $\omega_{\psi,V,W}(s,1)\cdot(\phi_1\otimes \phi_2)(w,x)= \psi(\frac{1}{2}\la sw,w \ra_{W})\cdot (\phi_1\otimes \phi_2)(w,x)  \text{ for } s \in \N_n(F) \ss \Hom(Y_n^*,Y_n)(F),$
 \end{itemize}
where $\bar{h} \text{ is the projection of } h \text{ to } \J_n(F)$ and $z^*\in \GL(Y_n^*)(F),\ t^* \in \Hom(Y_n^*,e)(F)$ are adjoint maps of $z$ and $t$, respectively.

For convenience, we identify $Y_n^* \otimes W$ with $W^n$ using the basis $\{e_{n}^*,\cdots,e_1^*\}$ of $Y_n^*$ and $e \otimes X_n$ with $X_n$. \\ Let $\M_{n\times n}$ be the $(n \times n)$ matrix group. Using the basis $\{e_1,\cdots,e_n\}$ of $Y_n$, $\{e_n^*,\cdots,e_1^*\}$ of $Y_n^*$ and $\{f_1,\cdots,f_n\}$ of $X_n$, we can identify $\Z_n(F)$, $\Hom(e,Y_n)$ and $\N_n(F)$ as a subgroup of $\GL_n(F)$, $F^{n}$ and $\M_{n\times n}(F)$, respectively.
%such that \[\N_n(F) \simeq \{ \begin{pmatrix} I_n & 0&s \\ & 1& 0 \\ & & I_n\end{pmatrix}\in \H_{n}(F)\}.\]

\noindent Let $\varpi$ be  the $(n \times n)$ anti-diagonal matrix whose non-zero elements are all one. With the above identifications, we can describe the action of $\Z_n(F)\cdot \Hom(e,X_n) \cdot \N_n(F)\times \wt{\J_n}(F)$ on $\big(S(W^n)\otimes S(X_n)\big)(F)$ as follows. \par
For $\phi_1\otimes \phi_2 \in \big(S(W^n)\otimes S(X_n)\big)(F)$ and $(w_1,\cdots,w_n;x) \in (W^n \oplus X_n)(F)$,
\begin{flalign}
 \label{a1}&\omega_{\psi,V,W}(1,h)(\phi_1\otimes \phi_2)(w_1,\cdots,w_n;x)=\phi_1(w_1\bar{h},\cdots,w_n\bar{h})\cdot (\omega_{\psi,W}(h)\phi_2)(x) \text{ for } h \in \wt{\J_n}(F),\\
 \label{a2} &\omega_{\psi,V,W}\big(z,1\big)(\phi_1\otimes \phi_2)(w_1,\cdots,w_n;x)=\phi_1\big((w_1,\cdots,w_n)\cdot \varpi z \varpi \big)\cdot \phi_2(x)  \text{ for } z \in \Z_n(F),\\
 \label{a3}&\omega_{\psi,V,W}(t,1)(\phi_1\otimes \phi_2)(w_1,\cdots,w_n;x)=\psi\big(\sum_{i=1}^n t_i\cdot\la w_{n+1-i},x\ra_W\big)(\phi_1\otimes \phi_2)(w_1,\cdots,w_n;x)  \text{ for } t\in F^n,\\
 \label{a4}&\omega_{\psi,V,W}(s,1)(\phi_1\otimes \phi_2)(w_1,\cdots,w_n;x)= \psi(\frac{1}{2}\tr\big(Gr(\mathbf{w})\cdot \varpi \cdot s\big))(\phi_1\otimes \phi_2)(w_1,\cdots,w_n;x)  \text{ for } s \in \N_n(F),
 \end{flalign}
where $\mathbf{w}=(w_1,v_2,\ldots,w_n), \ Gr(\mathbf{w})=\big(\la w_i,w_j \ra_{W}\big).$

Inspired by the proof of \cite[Proposition 2.1, Corollary 2.1]{JS03}, we now prove Lemma~\ref{tjw}. It is worth mentioning that \cite[Proposition 9.2]{GS} provides partial coverage of this for $\mathrm{char}(F)=0$ without presenting a proof.

\begin{proof}[Proof of Lemma~\ref{tjw}]
We choose the polarization $W^n \oplus X_n$ of $V \otimes W$ and utilize the action of the Weil representation with this polarization. Let us define
\[\overline{W_0}=\left\{\mathbf{w}=(w_1,\cdots,w_n) \in W^n  \ \vert \ Gr(\mathbf{w})=\begin{pmatrix} 0 & \cdots & 0 \\ \vdots  & \vdots & \vdots \\ 0 & \cdots & 0\end{pmatrix}\right\},
\]
where $ Gr(\mathbf{w})=\big(\la w_i,w_j\ra_{W}\big)$. We first claim that 
\[(\omega_{\psi,V_n,W_n})_{\U,\mu^{-1}}\simeq S(\overline{W_0}\oplus X_n).\]

Note that $\overline{W_0}$ is a closed subset of $W^n$. Therefore, by \cite{BZ}, we have the exact sequence
\[\xymatrix{0 \ar[r] & \mc{S}(W^n \oplus X_n  \bs \overline{W_0}\oplus X_n)  \ar[r]^-{\overline{\text{i}}} & \mc{S}(W^n \oplus X_n )  \ar[r]^-{\overline{\text{res}}} & \mc{S}(\overline{W_0}\oplus X_n) \ar[r] & 0},\]
where $\overline{\text{i}}$ is induced from the open inclusion map $i:W^n \oplus X_n  \bs \overline{W_0} \oplus X_n \to W^n \oplus X_n$ and $\overline{\text{res}}:\mc{S}(W^n \oplus X_n) \to \mc{S}(\overline{W_0}\oplus X_n)$ is the restriction map. Let $J_{\U,\mu^{-1}}$ be the twisted Jacquet functor with respect to $\U$ and $\mu^{-1}$. Since the functor $J_{\U,\mu^{-1}}$ is exact, we have the exact sequence
\[\xymatrix{0 \ar[r] & J_{\U,\mu^{-1}}\big(\mc{S}(W^n \oplus X_n  \bs \overline{W_0}\oplus X_n)\big)  \ar[r] & J_{\U,\mu^{-1}}\big(\mc{S}(W^n \oplus X_n)\big)  \ar[r] & J_{\U,\mu^{-1}}\big(\mc{S}(\overline{W_0}\oplus X_n)\big) \ar[r] & 0}.\]
By the definition of $\overline{W_{0}}$ and $(\ref{a4})$, $J_{\U,\mu^{-1}}\big(\mc{S}(W^n \oplus X_n  \bs \overline{W_0}\oplus X_n)\big)=0$ and so $J_{\U,\mu^{-1}}\big(\mc{S}(W^n \oplus X_n)\big)=\mc{S}(\overline{W_0}\oplus X_n)$. 

Therefore, our first claim is proved. 
Note that there is an action of $\Z_n(F)\times \wt{\J_n}(F)$ on $\overline{W_0}$ inherited from $W^n$ as follows;
\[(w_1,\cdots,w_n)\cdot (z,h)=((w_1\bar{h},\cdots,w_n\bar{h})\varpi z \varpi).
\]
Using the $\Z_n(F)$-action on $\overline{W_0}$, we can choose the representatives of the orbits of $\overline{W_0}$ as the form
\[(0,\cdots,0,w_j,0,\cdots,0,w_{j-1},0,\cdots,0,w_1,0,\cdots,0)\in \overline{W_0} \subset  W^n,
\]
for some $0 \leq j \leq n$, where $\{x_1,\cdots,x_j\}$ is a linearly independent set. 
By the Witt extension theorem, we can choose more restrictive representatives of the $(\Z_n(F)\times \wt{\J_n}(F))$-orbits of $\overline{W_0}$ as 
\[(0,\cdots,0,f_j^*,0,\cdots,0,f_{j-1}^{*},0,\cdots,0,f_1^*,0,\cdots,0)\in \overline{W_0}   \subset W^n.
\]
(Here, $0\le j \le n-1$ and we set $f_0^*=0$.)
Therefore, there are finite $(\Z_n(F)\times \wt{\J_n}(F))$-orbits in $\overline{W_0}$ of the forms
\[[(0,\cdots,0,f_j^*,0,\cdots,0,f_{j-1}^{*},0,\cdots,0,f_1^*,0,\cdots,0)] \subset W^n
\] and index them by $\{\overline{W_0}(i)\}_{1\le i \le N}$ so that $\dim(\overline{W_0}(i)) \le \dim(\overline{W_0}(j))$ for $i \le j$.

\par Note that for each $j\ge 1$, $\overline{W_0}(j)$ is a closed subset of $\bigcup_{i \ge j}\overline{W_0}(i)$ and therefore, we have the exact sequence
\begin{align} \label{ex}\xymatrix{0 \ar[r] & \mc{S}\big((\bigcup_{i \ge j+1}\overline{W_0}(i)) \big)  \ar[r] & \mc{S}\big((\bigcup_{i \ge j}\overline{W_0}(i)) \big)  \ar[r] & \mc{S}(\overline{W_0}(j)) \ar[r] & 0}.\end{align}
We claim that the Schwartz space on each orbits $\overline{W_0} $  whose representative is of the form \[(0,\cdots,0,f_k^*,0,\cdots,0,f_{k-1}^{*},0,\cdots,0,f_1^*,0,\cdots,0), \quad \text{for $k<n$}\] is zero.\par
Let $\overline{W_0}(j)$ be an orbit in $\overline{W_0}$ whose representative is $\bar{\mathbf{w}}=(0,\cdots,0,f_k^*,0,\cdots,0,f_{k-1}^{*},0,\cdots,0,f_1^*,0,\cdots,0)$ for some $k<n$. Suppose that $\mc{S}(\overline{W_0}(j))$ is non-zero and put $R_{\bar{\mathbf{w}}}$ the stabilizer of $\bar{\mathbf{w}}$ in $\Z_n(F)\times \wt{\J_n}(F)$.
Consider a map 
\[\Phi_{\bar{\textbf{w}}} :\mc{S}(\overline{W_{0}}(j)) \to \ind_{R_{\bar{\mathbf{v}}}}^{\Z_n \times \wt{\J_n}}\mathbb{I}, \ \varphi \mapsto \Phi_{\bar{\textbf{w}}}(\varphi),
\]
where $\Phi_{\bar{\textbf{w}}}$ is defined by
\[\Phi_{\bar{\textbf{w}}}(\varphi)(g,h)\coloneqq (\omega_{\psi,V,W}(g,h)\varphi)(\bar{\textbf{w}}).
\]
It can be readily verified that $\Phi_{\bar{\textbf{w}}}$ is a $(\Z_n(F) \times \wt{\J_n}(F))$-isomorphism. Consequently, we have $\mathcal{S}(\overline{W_0}(j)) \simeq \text{ind}_{R_{\bar{\mathbf{w}}}}^{\Z_n(F)\times \wt{\J_n}(F)}\mathbb{I}$ as a $(\Z_n(F) \times \wt{\J_n}(F))$-module. Since $k<n$, there is a simple root subgroup $J$ of $\Z_n(F)$ such that $J \times 1$ is a subgroup of $R_{\bar{\mathbf{w}}}$. However, $\mu$ is non-trivial on $J$ and it leads to a contradiction.\par
Therefore, by applying the exact sequence (\ref{ex}) repeatedly, we have
\[\mc{S}(W^n \oplus X_n)\simeq  \mc{S}(\overline{W_0}) \otimes  \mc{S}(X_n) \simeq \mc{S}(\overline{W_0}(N)) \otimes  \mc{S}(X_n),
\]
where $\overline{W_0}(N)$ is a $(\Z_n(F)\times \wt{\J_n}(F))$-orbit in $\overline{W_0}$ whose representative is $\bar{\mathbf{w}}'=(f_n^*,\cdots,f_{1}^*)$. 

By (\ref{a3}), $J_{\U,\mu^{-1}}\big(\mc{S}(\overline{W_0}(N)) \otimes  \mc{S}(X_n\bs \{f_n\})\big)=0$. Therefore, using a similar argument as before, we have 
\[J_{\U,\mu^{-1}}\big(\mc{S}(W^n \oplus X_n)\big)=J_{\U,\mu^{-1}}\big(\mc{S}(\overline{W_0}(N)) \otimes \mc{S}(\{f_n\})\big)=\mc{S}(\overline{W_0}(N)) \otimes \mc{S}(\{f_n\}).
\]
Let $R$ be the stabilizer of $\bar{\mathbf{w}}'$ in $\Z_n(F)\times \wt{\J_n}(F)$. Then,
\[R=\Biggr\{ \Biggl(  z, \begin{pmatrix}z & x \\ 
 & z^*\end{pmatrix}\Biggr) \in \Z_n(F)\times \wt{\J_n}(F)  \Biggl\},
\]
where we described the elements of $\wt{\J_n}$ using the basis $\{f_1,\cdots,f_{n},f_n^*,\cdots,f_{1}^*\}$.\par

%Put \[R_{f_n}^0=\Biggr\{ \Biggl(  z, \begin{pmatrix}z &  \\ 
 %& z^*\end{pmatrix}\Biggr) \in \Z_n\times \wt{\J_n}  \Biggl\}, \ R_{f_n}^1=\Biggr\{ \Biggl(  1, \begin{pmatrix}1 & x \\ 
%0 & 1\end{pmatrix}\Biggr) \in \Z_n\times \wt{\J_n}  \Biggl\}
%\]
%so that $R_{f_n}=R_{f_n}^0 \times R_{f_n}^1$.

Note that for $\phi_2 \in \mc{S}(X_n)$, $(\omega_{\psi,W}\Biggr( \begin{pmatrix}1 & x \\ 
 & 1\end{pmatrix}\Biggr)\phi_2)(f_n)=\psi(\frac{x_{n,1}}{2})\cdot \phi_2(f_n)$. Therefore, by $(\ref{a1}), (\ref{a2})$, we have \[\mc{S}(\overline{W_0}(N)) \otimes \mc{S}(\{f_n\})\simeq \ind_{R}^{\Z_n(F)\times \wt{\J_n}(F)}(\mathbb{I}\otimes \psi_n),\] 
where $\psi_n:R \to \CC^{\times}$ is defined by
\[\psi_n\Biggl(   \Biggl(  z, \begin{pmatrix}z & x \\ 
 & z^*\end{pmatrix}\Biggr) \Biggr)=\psi(\frac{x_{n,1}}{2}).
\]
Therefore, for $\phi \in \ind_{R}^{\Z_n(F)\times \wt{\J_n}(F)}(\mathbb{I}\otimes \psi_n)$ and $u'=\begin{pmatrix}z & x \\ 
 & z^*\end{pmatrix}\in \U'(F)$ we have
\[\phi\big(( 1,u'\cdot h)\big)=\psi(z_{1,2}+\cdots+z_{n-1,n}+\frac{x_{n,1}}{2})\phi((1,h))=\mu_1'(u')\phi\big((1,h)\big) \quad \text{for all } h \in \wt{\J_n}(F),\]
where $1$ is the identity element in $\Z_n(F)$. This proves our claim.
\end{proof}
\end{section}
\newpage

\section{The proof}

In this section, we will prove Theorem~\ref{main1}. To do so, we will make use of the Local Converse Theorem (LCT) for $\H_n^{+}$. Before proceeding, we will recall a result by Jo that improves upon the earlier work of Jiang and Sourdry (\cite{JS03}).

\begin{thm}[{\cite[Theorem~{2.11}]{Jo}}] \label{SO} Let $\pi$ and $\pi'$ be irreducible admissible generic representations of $\H_n^{+}(F)$ such that
\[\gamma(s,\pi \times \sigma,\psi)=\gamma(s,\pi' \times \sigma,\psi)
\] holds for any irreducible generic supercuspidal representation $\sigma$ of $\GL_i(F)$ with $1\le i \le n$. Then
\[\pi \simeq \pi'.
\]
\end{thm}
\begin{rem}In \cite[Theorem~{2.11}]{Jo}, an assumption is made regarding the equality of central characters between $\pi$ and $\pi'$. Nonetheless, considering that the center of $\H_n^+(F)$ is the trivial group, the central characters of $\pi$ and $\pi'$ are inherently trivial. Therefore, we can safely omit this assumption.

\end{rem}

The proof of Theorem~\ref{main1} consists of two steps, with the first step focusing on the tempered cases.

\begin{thm}\label{sc}
Let $\wt{\pi}$ and $\wt{\pi}'$ be irreducible admissible $\mu_1'$-generic tempered representations of $\wt{\J_n}(F)$ such that for any irreducible supercuspidal representation $\sigma$ of $\GL_i(F)$ with $1\le i \le n$, the equality
\[\gamma(s,\wt{\pi} \times \sigma,\psi)=\gamma(s,\wt{\pi}' \times \sigma,\psi)
\] holds. Then we have
\[\wt{\pi} \simeq \wt{\pi}'.
\]
\end{thm}

\begin{proof}
We begin by considering the bijection $\Theta_{\psi}$ and applying it to $\wt{\pi}$ and $\wt{\pi}'$. Since both $\wt{\pi}$ and $\wt{\pi}'$ are $\mu_1'$-generic and tempered, according to Theorem~\ref{Theta1}, the representations $\Theta_{\psi}(\wt{\pi})$ and $\Theta_{\psi}(\wt{\pi}')$ are irreducible generic tempered representations of $\H_n^{+}(F)$. Furthermore, by Theorem~\ref{Theta}, we have
\[\gamma(s,\Theta_{\psi}(\wt{\pi}) \times \sigma,\psi)=\gamma(s,\wt{\pi} \times \sigma,\psi)=\gamma(s,\wt{\pi}' \times \sigma,\psi)=\gamma(s,\Theta_{\psi}(\wt{\pi}')\times \sigma,\psi).
\]
for all irreducible supercuspidal representations $\sigma$ of $\GL_i(F)$.

Applying Theorem~\ref{SO}, we conclude that $\Theta_{\psi}(\wt{\pi}) \simeq \Theta_{\psi}(\wt{\pi}')$. Since $\Theta_{\psi}$ is a bijective correspondence, this implies that $\wt{\pi} \simeq \wt{\pi}'$.
\end{proof}

The following lemma is an extended version of \cite[Proposition 3.2]{JS03} and \cite[Proposition~{2.8} $(\rm{iii})$]{Jo} from supercuspidal to tempered representations.

\begin{lem} \label{temp}Let $\pi$ be an irreducible generic tempered representation of $\H_n^+(F)$ and $\sigma$ an irreducible unitary supercuspidal representation of $\GL_k(F)$. If $\gamma(s,\pi \times \sigma,\psi)$ has a real pole at $s=s_0$, then the pole must be a simple pole at $s_0=1$ and $\sigma \simeq \sigma^{\vee}$. 
\end{lem}

\begin{proof}

In the proof of \cite[Proposition~{2.8} $(\rm{iii})$]{Jo}, the requirement that $\pi$ is supercuspidal there was imposed in order to apply the `Casselman--Shahidi' Lemma (\cite[Lemma 2.6]{Jo}). However, it has been established in \cite[Proposition 12.3]{Kim} that the `Casselman--Shahidi Lemma' also holds for tempered representations. Therefore, except for this specific assumption, all other arguments presented in \cite[Proposition 2.8 $(\rm{iii})$]{Jo} can be equally applied to tempered representations. This completes the proof.  
\end{proof}

\begin{prop} 
\label{Pole}
Let $\sigma, \tau_i$ be  irreducible unitary supercuspidal representations of $\GL_k(F), \GL_{n_i}(F)$ for $1 \le i \le t$. Let $s_i$ be real numbers and $\wt{\pi_0}$ an irreducible $\mu_1'$-generic  tempered representation of $\wt{\J_n}(F)$. Then we have the following;
\begin{enumerate}[label=$(\mathrm{\roman*})$]
\item If the product $\prod_{i=1}^t \gamma(s+s_i,\tau_i \times \sigma,\psi)$ has a real pole $($respectively, a real zero$)$ at $s=s_0$, then $\sigma \cong \tau_i^{\vee}$
and $s_0=1-s_i$ $($respectively, $s_0=-s_i)$ for some $1 \leq i \leq t$.
\item If the product $\prod_{i=1}^t \gamma(s-s_i,\tau_i^{\vee} \times \sigma,\psi)$ has a real pole $($respectively, a real zero$)$ at $s=s_0$, then $\sigma \cong \tau_i$
and $s_0=1+s_i$ $($respectively, $s_0=s_i)$ for some $1 \leq i \leq t$.
%\item For each $j=1,2,\dotsm,t$, the product $\prod_{i=1}^t \gamma(s+z_i,\sigma_i^{\vee} \times \sigma_j,\psi)\gamma(s-z_i,\sigma_i \times \sigma_j,\psi)$ 
\item\label{PoleDecomposition-3} The factor $\gamma(s,\cl{\pi}_0 \times \sigma,\psi)$ has no zero for ${\Re}(s) > 0$. 
 If $\gamma(s,\wt{\pi_0} \times \sigma,\psi)$ has a real pole at $s=s_0$, then the pole must be a simple pole at $s_0=1$ and $\sigma \cong \sigma^{\vee}$. 
\end{enumerate}
\end{prop}

\begin{proof}$(\rm i)$ and $(\rm{ii})$ follows from the discussion of \cite[Section 3.2]{JS03}. The zeros of $\gamma(s,\wt{\pi_0} \times \sigma,\psi)$ arises from the poles of $L_{\psi_0}(s,\wt{\pi_0} \times \sigma)$. Therefore, the first part of $(\rm{iii})$ follows from Corollary~\ref{chol}.
For the second statement of (iii),
consider $\Theta_{\psi}(\wt{\pi_0})$. By Theorem~\ref{tem0} and Theorem~\ref{Theta}, $\Theta_{\psi}(\wt{\pi_0})$ is tempered and generic. Therefore,
\[\gamma(s,\wt{\pi_0} \times \sigma,\psi)=\gamma(s,\Theta_{\psi}(\wt{\pi_0}) \times \sigma,\psi)
\] and it follows from Lemma~\ref{temp} $(\rm{iii})$.
\end{proof}

The proposition stated below is a consequence of Proposition~\ref{Pole}. Although the proof is similar to that of \cite[Theorem 5.1]{JS03}, due to its importance, we provide the proof.

\begin{prop}
\label{equal}
 Let $\wt{\pi}$ (resp. $\wt{\pi}'$) be an irreducible $\mu_1'$-generic representations of $\wt{\J_n}(F)$, which has tempered support $(\tau_1,\cdots,\tau_r;\wt{\pi_0})$ (resp. $(\tau_1',\cdots,\tau_{r'}';\wt{\pi_0}')$) and exponents $(s_1,\cdots,s_r)$ (resp. $(s_1',\cdots,s_{r'}')$).
 Suppose that
\begin{multline}
\label{GeneralDec}
\left[\prod_{i=1}^r \gamma(s+s_i,\tau_i \times \sigma,\psi)\gamma(s-s_i,\tau_i^{\vee} \times \sigma,\psi) \right]\gamma(s,\wt{\pi_0} \times \sigma,\psi)\\
 =\left[\prod_{i=1}^{r'} \gamma(s+s_i',\tau_i' \times \sigma,\psi)\gamma(s-s_i',(\tau_i')^{\vee} \times \sigma,\psi) \right]\gamma(s,\wt{\pi_0}' \times \sigma,\psi)
\end{multline}
for all irreducible unitary supercuspidal representations $\sigma$ of ${\GL}_t(F)$ with $1 \leq t \leq n$. Then $r=r'$ and there exists a permutation $\mathfrak{p}$
of $\{ 1,2,\dotsm,r\}$ such that
\begin{enumerate}[label=$(\mathrm{\roman*})$]
\item $s_i=s'_{\mathfrak{p}(i)}$ and $\tau_i \cong \tau_{\mathfrak{p}(i)}'$ for all $i=1,2,\dotsm,r$;
\item $\gamma(s,\wt{\pi_0} \times \sigma,\psi)=\gamma(s,\wt{\pi_0}' \times \sigma,\psi)$ for all irreducible unitary supercuspidal representations $\sigma$ of ${\GL}_t(F)$ with $1 \leq t \leq n$. 
\end{enumerate}
\end{prop}
\begin{proof}
Since the equality (\ref{GeneralDec}) holds for any irreducible supercuspidal representation $\sigma$ of $\GL_t(F)$, put $\sigma=\tau_1$. Then, we have \begin{multline}
\label{GD1}
\left[\prod_{i=1}^r \gamma(s+s_i,\tau_i \times \tau_1,\psi)\cdot \gamma(s-s_i,\tau_i^{\vee} \times \tau_1,\psi) \right] \cdot\gamma(s,\wt{\pi_{0}} \times \tau_1,\psi)\\
 =\left[\prod_{i=1}^{r'} \gamma(s+s_i',\tau_i' \times \tau_1,\psi)\cdot \gamma(s-s_i',(\tau_i')^{\vee} \times \tau_1,\psi) \right] \cdot\gamma(s,\wt{\pi_{0}}' \times \tau_1,\psi).
\end{multline}
By Proposition~\ref{temp}, $\gamma(s-s_1,\tau_1^{\vee}\times \tau_1
,\psi)$ has a pole at $s=s_1+1$ and the left-hand side (LHS) of the equation (\ref{GD1}) has no zero at $s=s_1+1$. Therefore, it has a pole at $s=s_1+1$. The poles on the right-hand side (RHS) of the equation (\ref{GD1}) can arise from one of the following terms:
\begin{enumerate}
    \item $\prod_{1\le i \le r'}\gamma(s+s_i',\tau_i' \times \tau_1,\psi)$,
    \item $\prod_{1\le i \le r'}\gamma(s-s_i',(\tau_i')^{\vee} \times \tau_1,\psi)$; or 
    \item $\gamma(s,\wt{\pi_{0}}' \times \tau_1,\psi)$. 
\end{enumerate}

If the pole $s=s_1+1$ on the RHS originates from $\prod_{1\le i \le r'}\gamma(s-s_i',(\tau_i')^{\vee} \times \tau_1,\psi)$, then by Proposition~\ref{temp} (ii), we must have $s_1+1=s_i'+1$ and $\tau_1=\tau_i'$ for some $1\le i \le r'$. Consequently, we can cancel the term
\[
\gamma(s+s_1,\tau_1 \times \sigma,\psi)\cdot \gamma(s-s_1,\tau_1^{\vee} \times \sigma,\psi)
=
\gamma(s+s_i',\tau_i' \times \sigma,\psi)\cdot \gamma(s-s_i',(\tau_i')^{\vee} \times \sigma,\psi)
\]
on both sides of the equation (\ref{GeneralDec}). 

For later use in the proof, we call the above argument that cancels gamma factors as Argument A. 

By applying the Argument A iteratively for $s=s_i+1$, there exist some $1\le t \le r$ and a permutation $\mathfrak{p}$ of $\{1,2,\cdots,r'\}$ such that $s=s_j +1$ is a pole of $\prod_{i=1}^{r'-(t-1)}\gamma(s-s_{\mathfrak{p}(i)}',(\tau_{\mathfrak{p}(i)}')^{\vee} \times \tau_t,\psi)$ for $j=1, \cdots, t-1$ and $s=s_t+1$ is not a pole of $\prod_{i=1}^{r'-(t-1)}\gamma(s-s_{\mathfrak{p}(i)}',(\tau_{\mathfrak{p}(i)}')^{\vee} \times \tau_t,\psi)$. Thus, we obtain the following refined equality:
\begin{multline}
\label{GD2}
\left[\prod_{i=t}^r \gamma(s+s_i,\tau_i \times \sigma,\psi)\cdot \gamma(s-s_i,\tau_i^{\vee} \times \sigma,\psi) \right] \cdot\gamma(s,\wt{\pi_{0}} \times \sigma,\psi)\\
 =\left[\prod_{i=1}^{r'-(t-1)} \gamma(s+s_{\mathfrak{p}(i)}',\tau_{\mathfrak{p}(i)}' \times \sigma,\psi)\cdot \gamma(s-s_{\mathfrak{p}(i)}',(\tau_{\mathfrak{p}(i)}')^{\vee} \times \sigma,\psi) \right] \cdot\gamma(s,\wt{\pi_{0}}' \times \sigma,\psi)
\end{multline}
Note that $s=1+s_t$ is a pole of either $\prod_{i=1}^{r'-(t-1)} \gamma(s+s_{\mathfrak{p}(i)}',\tau_{\mathfrak{p}(i)}' \times \tau_t,\psi)$  or $\gamma(s,\wt{\pi_0}'\times \tau_t,\psi)$.

If $s=1+s_t$ is a pole of $\gamma(s,\wt{\pi_{0}}'\times \tau_t,\psi)$, then Proposition~{\ref{temp}} (iv) implies that $s_t=0$ and $\tau_t=\tau_t^{\vee}$. %Then $s_t =s_{t+1} \cdots = s_r=0$ and thereby (\ref{GD2}) turns into 
%\begin{multline}
%\label{GD3}
%\left[\gamma(s,\rho_t \times \rho_t,\psi)^2\cdot \prod_{i=t+1}^r \gamma(s,\rho_i \times \rho_t,\psi)\cdot \gamma(s,\rho_i^{\vee} \times \rho_t,\psi) \right] \cdot\gamma(s,\wt{\pi_{01}} \times \rho_t,\psi)\\
 %=\left[\prod_{i=1}^{r'-(t-1)} \gamma(s+s_{\mathfrak{p}(i)}',\rho_{\mathfrak{p}(i)}' \times \rho_t,\psi)\cdot \gamma(s-s_{\mathfrak{p}(i)}',(\rho_{\mathfrak{p}(i)}')^{\vee} \times \rho_t,\psi) \right] \cdot\gamma(s,\wt{\pi_{02}} \times \rho_t,\psi).
%\end{multline}
Thus, $\gamma(s,\tau_t \times \tau_t,\psi)^2$ has a \textit{double pole} at $s=1$, while $\gamma(s,\wt{\pi_0}'\times \tau_t,\psi)$ has a simple pole at $s=1$. Again, by putting $\sigma=\tau_t$ in the equation (\ref{GD2}), we see that there exists a $1\le i \le r'-(t-1)$ such that $\gamma(s+s_{\mathfrak{p}(i)}',\tau_{\mathfrak{p}(i)}' \times \tau_t,\psi)$ has a pole at $s=1$, which implies that $s_{\mathfrak{p}(i)}'=0$ and $\tau_t=(\tau_{\mathfrak{p}(i)}')^{\vee}$. Therefore, if $s=1+s_t$ is a pole of $\gamma(s,\wt{\pi}' \times \tau_t,\psi)$, then we can cancel out $\gamma(s+s_t,\tau_t \times \sigma,\psi)\cdot \gamma(s-s_t,\tau_t^{\vee} \times \sigma,\psi)=\gamma(s+s_{\mathfrak{p}(i)}',\tau_{\mathfrak{p}(i)}' \times \sigma,\psi)\cdot \gamma(s-s_{\mathfrak{p}(i)}',(\tau_{\mathfrak{p}(i)}')^{\vee} \times \sigma,\psi)$ on both sides of the equation (\ref{GD2}).

Similarly, for later use in the proof, in case $s=1+s_t$ is a pole of $\gamma(s,\wt{\pi_{0}}'\times \tau_t,\psi)$, we call the above argument that cancels gamma factors as Argument B. 

By applying Arguments A and B iteratively, we eventually obtain the following equality:
\begin{multline}
\label{GD4}
\left[\prod_{i=t_1}^r \gamma(s+s_i,\tau_i \times \sigma,\psi)\cdot \gamma(s-s_i,\tau_i^{\vee} \times \sigma,\psi) \right] \cdot\gamma(s,\wt{\pi_{0}} \times \sigma,\psi)\\
 =\left[\prod_{i=1}^{r'-(t_1-1)} \gamma(s+s_{\mathfrak{p}_1(i)}',\tau_{\mathfrak{p}_1(i)}' \times \sigma,\psi)\cdot \gamma(s-s_{\mathfrak{p}_1(i)}',(\tau_{\mathfrak{p}_1(i)}')^{\vee} \times \sigma,\psi) \right] \cdot\gamma(s,\wt{\pi_{0}}' \times \sigma,\psi)
\end{multline} for some $t \le t_1 \le r$ and a permutation $\mathfrak{p}_1$ of $\{1,2,\cdots,r'\}$ such that $s=1+s_{t_1}$ is not a pole of $\prod_{i=1}^{r'-(t_1-1)}  \gamma(s-s_{\mathfrak{p}_1(i)}',(\tau_{\mathfrak{p}_1(i)}')^{\vee} \times \tau_{t_1},\psi)$ and $\gamma(s,\wt{\pi_{0}}'\times \tau_{t_1},\psi)$.

Since $s=1+s_{t_1}$ is a pole of 
$\gamma(s-s_{t_1},\tau_{t_1}^{\vee} \times \tau_{t_1},\psi)$, it follows that the RHS of the equation (\ref{GD4}) with $\sigma=\tau_{t_1}$ has a pole at $s=1+s_{t_1}$. By our choice of $t_1$, it must be a pole of $\prod_{i=1}^{r'-(t_1-1)} \gamma(s+s_{\mathfrak{p}_1(i)}',\tau_{\mathfrak{p}_1(i)}' \times \tau_{t_1},\psi)$. By Proposition~\ref{temp}, we have $1+s_{t_1}=1-s_{\mathfrak{p}_1(i_1)}'$ and $\tau_{t_1}=(\tau_{\mathfrak{p}_1(i_1)}')^{\vee}$ for some $1\le i_1 \le r'-(t_1-1)$. Therefore, $s_{t_1}=s_{\mathfrak{p}_1(i_1)}'=0$. If $s_{\mathfrak{p}_1(k)}'>0$ for some $1\le k \le r'-(t_1-1)$, then the LHS of the equation $(\ref{GD4})$ with $\sigma=\tau_{\mathfrak{p}_1(k)}'$ has a pole at $s=1+s_{\mathfrak{p}_1(k)}'$. However, this is impossible because $s_{t_1}=s_{t_1+1}=\cdots=s_r=0$. Therefore, $s_{\mathfrak{p}_1(i)}'=0$ for all $1\le i \le r'-(t_1-1)$ and we have:
\begin{multline}
\label{GD5}
\left[\prod_{i=t_1}^r \gamma(s,\tau_i \times \sigma,\psi)\cdot \gamma(s,\tau_i^{\vee} \times \sigma,\psi) \right] \cdot\gamma(s,\wt{\pi_{0}} \times \sigma,\psi)\\
 =\left[\prod_{i=1}^{r'-(t_1-1)} \gamma(s,\tau_{\mathfrak{p}_1(i)}' \times \sigma,\psi)\cdot \gamma(s,(\tau_{\mathfrak{p}_1(i)}')^{\vee} \times \sigma,\psi) \right] \cdot\gamma(s,\wt{\pi_{0}}' \times \sigma,\psi).
\end{multline} 
Since $\tau_{t_1}=(\tau_{\mathfrak{p}_1(i_1)}')^{\vee}$, we can remove $\gamma(s,\tau_{t_1} \times \sigma,\psi)\cdot \gamma(s,\tau_{t_1}^{\vee} \times \sigma,\psi)=\gamma(s,\tau_{\mathfrak{p}_1(i)}' \times \sigma,\psi)\cdot \gamma(s,(\tau_{\mathfrak{p}_1(i)}')^{\vee} \times \sigma,\psi)$ on both sides of the equation (\ref{GD5}). In this way, applying Arguments A, B, and the above argument, we can cancel out $\left[\prod_{i=t_1}^r \gamma(s,\tau_i \times \sigma,\psi)\cdot \gamma(s,\tau_i^{\vee} \times \sigma,\psi) \right]$ in \[\left[\prod_{i=1}^{r'-(t_1-1)} \gamma(s,\tau_{\mathfrak{p}_1(i)}' \times \sigma,\psi)\cdot \gamma(s,(\tau_{\mathfrak{p}_1(i)}')^{\vee} \times \sigma,\psi) \right]\] and we obtain
\begin{equation}
\label{GD6}
 \gamma(s,\wt{\pi_{0}} \times \sigma,\psi)=\left[\prod_{i=1}^{r'-r} \gamma(s,\tau_{\mathfrak{p}_2(i)}' \times \sigma,\psi)\cdot \gamma(s,(\tau_{\mathfrak{p}_2(i)}')^{\vee} \times \sigma,\psi) \right]\cdot\gamma(s,\wt{\pi_{0}}' \times \sigma,\psi)
\end{equation} for some permutation $\mathfrak{p}_2$ of $\{1,2,\cdots,r'\}$. If $r'>r$, then we put $\sigma=\tau_{\mathfrak{p}_2(1)}'$ in the equation (\ref{GD6}). It follows that $\gamma(s,(\tau_{\mathfrak{p}_2(1)}')^{\vee} \times \tau_{\mathfrak{p}_2(1)}',\psi)$ has a pole at $s=1$, and hence $ \gamma(s,\wt{\pi_{0}} \times \tau_{\mathfrak{p}_2(1)}',\psi)$ must also have a pole at $s=1$, which implies $\tau_{\mathfrak{p}_2(1)}'=(\tau_{\mathfrak{p}_2(1)}')^{\vee}$. However, while the RHS of the equation (\ref{GD6}) with $\sigma=\tau_{\mathfrak{p}_2(1)}'$ has at least double pole at $s=1$, $ \gamma(s,\wt{\pi_{0}} \times \tau_{\mathfrak{p}_2(1)}',\psi)$ has at most a simple pole at $s=1$. This is a contradiction, and thus we conclude that $r'=r$, completing the proof.
\end{proof}

Proposition~{\ref{equal}} immediately leads to the following corollary.
\begin{cor}\label{uniq}Let $\wt{\pi}$ be an irreducible $\mu_1'$-generic representation of $\wt{\J_n}(F)$. Then the tempered support and exponents of $\wt{\pi}$ are uniquely determined up to permutation.     
\end{cor}

We are now prepared to prove Theorem~\ref{main1}.
\begin{proof}
Assume that $\wt{\pi}$ and $\wt{\pi}'$ have tempered supports as stated in Proposition~\ref{equal}. By Proposition~\ref{gamma} (i), Theorem~\ref{sc}, and Proposition~\ref{equal}, we conclude that $\wt{\pi}_0 \simeq \wt{\pi}_0'$, which implies that $\wt{\pi}$ and $\wt{\pi}'$ have the same supports. Since an induced representation has a unique generic constituent, we have $\wt{\pi} \simeq \wt{\pi}'$.
\end{proof}

Using similar arguments employed thus far, we can establish Theorem~\ref{stab} as follows.

\begin{proof} Assume that $\wt{\pi}$ has tempered support $(\tau_1,\cdots,\tau_r;\wt{\pi}_0)$ and exponents $(s_1,\cdots,s_r)$, and $\wt{\pi}'$ has tempered support $(\tau_1',\cdots,\tau_{r'}';\wt{\pi}_0')$ and exponents $(s_1',\cdots,s_{r'}')$. Consider the irreducible generic representation $\pi$ of $\H_n^+(F)$, which has temepred support $(\tau_1,\cdots,\tau_r;\Theta_{\psi}(\wt{\pi}_0))$ and exponents $(s_1,\cdots,s_r)$, and $\pi'$ of $\H_n^+(F)$, which has support $(\tau_1',\cdots,\tau_{r'}';\Theta_{\psi}(\wt{\pi}_0'))$ and exponents $(s_1',\cdots,s_{r'}')$.

By \cite[Theorem 4.1]{CKPSS04}, \cite[Theorem]{CP}, and \cite[Theorem 1.4]{L15}, there exists an integer $l_0 = l(\pi,\pi') > 0$ such that for any quasi-character $\chi$ of $F^\times$ with $\text{cond}(\chi) > l_0$,
\[ \gamma(s,\pi \times \chi,\psi)=\gamma(s,\pi' \times \chi,\psi).\]

\noindent Furthermore, by Theorem~\ref{Theta}, we know that
\[ \gamma(s,\wt{\pi_0} \times \chi,\psi)=\gamma(s,\Theta_{\psi}(\wt{\pi_0}) \times \chi,\psi), \quad \gamma(s,\wt{\pi_0}' \times \chi,\psi)=\gamma(s,\Theta_{\psi}(\wt{\pi_0}') \times \chi,\psi).\]
Therefore, 
\[ \gamma(s,\wt{\pi} \times \chi,\psi)=\gamma(s,\pi \times \chi,\psi), \quad \gamma(s,\wt{\pi}' \times \chi,\psi)=\gamma(s,\pi' \times \chi,\psi).\]
By taking $l = l_0$, this completes the proof.
\end{proof}
\begin{rem}Indeed, the proof of Theorem~\ref{stab} does not essentially require tempered supports (i.e., $\wt{\pi}_0$ and $\wt{\pi}_0'$ being tempered), but supercuspidal support is sufficient.
\end{rem}
\newpage
\section{The rigidity theorem for $\Mp_{2n}(\A)$}
Let \( K \) be a number field and \( \A \) its ad\`{e}le ring. Let \( \WW_{n} \) be a symplectic space over \( K \) of dimension \( 2n \), and put \( \Mp_{2n} \coloneqq \Mp(\WW_n) \). We prove the rigidity theorem for \( \Mp_{2n}(\A) \) using the global theta lifts between metaplectic groups and odd special orthogonal groups.

Let \( \Psi = \otimes_v \Psi_v \) be a non-trivial additive character of \( K \backslash \A \), and let \( \mf{U}' \) be the generic character of \( \U'(K) \backslash \U'(\A) \) associated to \( \Psi \), defined similarly to \( \mu_{1}' \) in the local case.

Now we are ready to prove Theorem~\ref{rigid}.

\begin{proof}[Proof of Theorem~\ref{rigid}]
Choose \( \{\VV_{m}\}_{m \ge 1} \) a Witt tower of split quadratic spaces over \( K \) such that the dimension of \( \VV_{m} \) over \( K \) is \( 2m + 1 \). Put \( \SO_{2m+1} \coloneqq \SO(\VV_m) \). For each \( 1 \le i \le 2 \), consider the tower of global theta lifts \( \{ \Theta_{\Psi^{-1},\VV_{m},\WW_n}(\wt{\pi}_i) \}_{m \ge 1} \) of \( \wt{\pi}_i \) to \( \{\SO_{2m+1}(\A)\}_{m \ge 1} \) with respect to \( \Psi^{-1} \). Let \( n_i \) be the minimal integer \( k \) such that \( \Theta_{\Psi^{-1},\VV_{k},\WW_n}(\wt{\pi}_i) \) is nonzero. By \cite[Proposition~3]{Fu}, \( \Theta_{\Psi^{-1},\VV_{n},\WW_n}(\wt{\pi}_i) \) is nonzero and globally generic, so \( n_i \le n \). By Rallis's tower property, \( \Theta_{\Psi^{-1},\VV_{n_i},\WW_n}(\wt{\pi}_i) \) is cuspidal and hence irreducible (see \cite[Proposition~3.1]{Gan1}). Put \( \pi_i = \Theta_{\Psi^{-1},\VV_{n_i},\WW_n}(\wt{\pi}_i) \). Then for any place \( v \) of \( K \), \( \pi_{i,v} \) is the non-zero local theta lift \( \theta_{\Psi_v^{-1},(\VV_{n_i})_v,(\WW_n)_v}(\wt{\pi}_{i,v}) \) of \( \wt{\pi}_{i,v} \). If \( n_i < n-1 \), this is impossible by \cite[Corollary~2.2 (2)]{JS03}. Therefore, \( n_i \ge n-1 \), hence \( n_i = n-1 \) or \( n_i = n \). 

If \( n_i = n \), then \( \pi_i \) is globally generic and cuspidal. Suppose \( n_i = n-1 \). Since \( \wt{\pi}_i \) is irreducible, cuspidal, and \( \mf{U}' \)-generic, by \cite[Theorem~1.1]{JS07}, the global theta lift \( \Theta_{\Psi,\WW_n,\VV_{n_i}}(\pi_i) \) of \( \pi_i \) to \( \Mp_{2n}(\A) \) with respect to \( \Psi \) is \( \wt{\pi}_i \). Since \( \wt{\pi}_i \) is \( \mf{U}' \)-generic, by the relation of Whittaker–Fourier coefficients of \( \pi_i \) and \( \wt{\pi}_i = \Theta_{\Psi,\WW_n,\VV_{n_i}}(\pi_i) \) in \cite[(4.1), p.~744]{JS07}, \( \pi_i \) is globally generic. Therefore, even when \( n_i = n-1 \), \( \pi_i \) is globally generic and cuspidal.

Suppose \( n_1 \ne n_2 \). Choose a finite place \( v \) of \( K \) such that \( \wt{\pi}_{1,v} \simeq \wt{\pi}_{2,v} \) and both are unramified. For each \( 1 \le i \le 2 \), write \( \pi_{i,v} \) as a subquotient of an induced representation
\[
\Ind_{\P_{n_i}(K_v)}^{\SO_{2n_i+1}(K_v)}
(\chi_{i,1}|\det|_{K_v}^{s_{i,1}} \otimes \cdots \otimes \chi_{i,n_i}|\det|_{K_v}^{s_{i,n_i}}),
\]
where \( \P_{n_i} \) is a Borel subgroup of \( \SO_{2n_i+1} \) with Levi subgroup isomorphic to \( \GL_{1} \times \cdots \times \GL_{1} \), and for all \( 1 \le k \le n_i \), \( \chi_{i,k} \) are unitary unramified characters of \( \GL_1(K_v) \), and \( s_{i,k} \ge 0 \) are real numbers. Note that \( 0 \le s_{i,k} < \frac{1}{2} \) by the non-trivial Ramanujan bound \cite[Corollary~10.1]{CKPSS}.

Since \( \theta_{\Psi_v^{-1},(\VV_{n_i})_v,(\WW_n)_v}(\wt{\pi}_{i,v}) = \pi_{i,v} \), we may assume \( n_1 = n \), \( n_2 = n - 1 \). By Theorem~\ref{ksc}, \( \wt{\pi}_{1,v} \) (resp.~\( \wt{\pi}_{2,v} \)) is a subquotient of a unitary induced representation
\[
\Ind_{\cl{\Q_n}(K_v)}^{\Mp_{2n}(K_v)}
((\chi_{1,1}|\det|_{K_v}^{s_{1,1}})_\psi \otimes \cdots \otimes (\chi_{1,n}|\det|_{K_v}^{s_{1,n}})_\psi),
\]
\[
(\text{resp. } \Ind_{\cl{\Q_n}(K_v)}^{\Mp_{2n}(K_v)}
((|\det|_{K_v}^{1/2})_\psi \otimes (\chi_{2,1}|\det|_{K_v}^{s_{2,1}})_\psi \otimes \cdots \otimes (\chi_{2,n-1}|\det|_{K_v}^{s_{2,n-1}})_\psi),)
\]
where \( \cl{\Q_k} \) is a Borel subgroup of \( \Mp_{2k} \) with Levi subgroup isomorphic to \( \cl{\GL}_{1} \times_{\{\pm 1\}} \cdots \times_{\{\pm 1\}} \cl{\GL}_{k} \). However, since all exponents \( s_{i,j} \) are less than $\frac{1}{2}$, it contradicts the assumption \( \wt{\pi}_{1,v} \simeq \wt{\pi}_{2,v} \). The case \( n_1 = n - 1 \), \( n_2 = n \) is similar. Hence \( n_1 = n_2 \).

 Since \( \wt{\pi}_{1,v} \simeq \wt{\pi}_{2,v} \) for almost all places \( v \) of $K$, we have \( \pi_{1,v} \simeq \pi_{2,v} \) for almost all places \( v \) of $K$. By the rigidity theorem for $\SO_{2n_i+1}(\A)$ (see \cite[Theorem~5.3]{JS03}), we have \( \pi_1 \simeq \pi_2 \). Theorem~{\ref{howe}} asserts the injectivity of the global theta correspondence; hence, we deduce \( \wt{\pi}_1 \simeq \wt{\pi}_2 \), completing the proof.
\end{proof}

\subsection*{Acknowledgements} The author would like to thank Hiraku Atobe and Eyal Kaplan for their valuable comments. He also extends heartfelt thanks to Dongho Byeon, Youngju Choie, Youn-Seo Choi, Wee Teck Gan, Haseo Ki, and Sug Woo Shin for their unwavering support and encouragement over the years. The author is grateful to KAIST, Yonsei University, and Catholic Kwandong University for their excellent research environment and generous support. Finally, he is deeply thankful to the referee for detailed and insightful suggestions, which undoubtedly improved the exposition of this paper.

This work is supported by the National Research Foundation of Korea (NRF) grant funded by the Korean government (MSIT) under grant number 2020R1F1A1A01048645 and RS-2023-00237811.

\subsection*{Data availability} No datasets were generated or analyzed during the current study.

\subsection*{Conflict of interest}
The author declares that there is no conflict of interest.

\providecommand{\bysame}{\leavevmode\hbox to3em{\hrulefill}\thinspace}


\begin{thebibliography}{99}

\bibitem[AB95]{AB95}
{J. Adams and D. M. Barbasch},
{\em Reductive Dual Pair Correspondence for Complex Groups},
{\it J. Funct. Anal.} {\bf 132} (1995) 1--42.

\bibitem[Ar13]{Ar13}
{J. Arthur},
{\em The endoscopic classification of representations. Orthogonal and symplectic groups}, 
{\it Amer. Math. Soc. Colloq. Publ.} {\bf 61} (2013).

\bibitem[At]{At} H. Atobe, \emph{The local theta correspondence and the local Gan–Gross–Prasad conjecture for the symplectic-metaplectic case}, Math. Annalen., \textbf{371}, (2018), 225--295.

\bibitem[BZ]{BZ}
J. H. Bernstein, A. V. Zelevinskii, 
{\em Representations of the group GL(n,F) where F is a non-Archimedean local field},
{\it Russian Mathematical Surveys}
{\bf 31}, Issue 3 (1976) 1–68.

\bibitem[CFK22]{CFK22} 
Y. Cai, S. Friedberg, and E. Kaplan, 
\emph{The generalized doubling method: local theory}, 
{\it Geom. Funct. Anal.} {\bf 32} (2022) 1233--1333.

\bibitem[CS]{CS} W. Casselman and F. Shahidi, \emph{On irreducibility of standard modules for generic representations}, Ann. Sci. \'{E}cole Norm. Sup. \textbf{31} (1998), 561--589.

 \bibitem[Cha19]{Cha19} Jingsong Chai, \emph{Bessel functions and local converse conjecture of Jacquet}, J. Eur. Math. Soc., \textbf{21}, (2019), no. 6, 1703--1728.

 \bibitem[CKPSS]{CKPSS} J. W. Cogdell, H. Kim, I.I. Piathtski-Shapiro, F. Shahidi, \emph{On lifting from classical groups to $\GL(n)$}, IHES Publ. Math. \textbf{93} (2001), 5--30. 

 \bibitem[CKPSS04]{CKPSS04} J. W. Cogdell, H. Kim, I.I. Piathtski-Shapiro, F. Shahidi, \emph{Functoriality for the classical groups}, IHES Publ. Math. \textbf{99} (2004), 163--233. 

 \bibitem[CP]{CP} J. W. Cogdell, I.I. Piathtski-Shapiro, \emph{Stability of gamma factors for $\SO(2n+1)$}, Manuscr. Math., \textbf{95} (1998), 437--461. 

\bibitem[CGL24]{CGL24}
H. Castillo, G. Henniart, L. Lomelí,
{\em On generic representations of quasi-split reductive groups over local fields of positive characteristic}
\href{https://arxiv.org/abs/2412.00229v1}{arXiv:2412.00229v1}

\bibitem[CST17]{CST17} J. W. Cogdell, F. Shahidi and T.-L. Tsai, 
  \emph{Local Langlands correspondence for $\GL_n$ and the exterior
   and symmetric square $\varepsilon$-factors}, Duke Math. J. \textbf{166},
   (2017), no. 11, 2053--2132.

\bibitem[Fu]{Fu} M. Furusawa, \emph{On the theta lift from $\SO_{2n+1}$ to $\Mp_n$}, (1995), \textbf{466}, 87-110.

\bibitem[Gan1]{Gan1}
{W. T. Gan}
{\em Automorphic Forms and the Theta Correspondence},
\href{https://arxiv.org/abs/2303.14918}{[arXiv:2303.14918]}

\bibitem[Gan2]{Gan2} W. T. Gan, \emph{Doubling zeta integrals and local factors for metaplectic groups}, Nagoya Math. J., \textbf{208}, (2012), 67 - 95.

\bibitem[GGP]{GGP} W. T. Gan, B. Gross, D. Prasad, \emph{Symplectic local root numbers, central critical $L$ values, and
   restriction problems in the representation theory of classical groups}, Ast\'{e}risque., \textbf{346}, (2012), 1--109.

\bibitem[GI1]{GI1}
{W. T. Gan and A. Ichino}, 
{\em Formal degrees and local theta correspondence}, 
{\it Invent. Math.} {\bf 195} (2014) no. 3, 509--672. 

\bibitem[GI2]{GI2} {W. T. Gan and A. Ichino}, \emph{The Gross--Prasad conjecture and local theta correspondence}, Invent. Math., \textbf{206}, (2016), no.~3, 705--799.

\bibitem[GI3]{GI3}
{W. T. Gan and A. Ichino}, 
{\em The Shimura–Waldspurger correspondence for $Mp_{2n}$}, 
{\it Ann. of Math.} {\bf 188} (2018) no. 3, 965--1016. 


\bibitem[GL]{GL}
{W. T. Gan and L. Lomelí}, 
{\em Globalization of supercuspidal representations over function fields and applications},
{\it J. Eur. Math. Soc.} {\bf 20} (2018), no. 11, 2813-–2858.

\bibitem[GS]{GS} W. T. Gan and G. Savin, \emph{Representations of metaplectic groups $\mathrm{I}$: epsilon dichotomy and local Langlands correspondence}, Composito, \textbf{148} (2012), 1655--1694.

\bibitem[GRS]{GRS}
D. Ginzburg, S. Rallis and D. Soudry, 
\emph{The descent map from automorphic representations of $GL(n)$ to classical groups}, 
World Scientific, Hackensack (2011).

%\bibitem[GRS98]{GRS98} D. Ginzberg, S. Rallis, D. Soudry, \emph{$L$-functions for symplectic groups},  Bull. Soc. Math. France, \textbf{126}, (1998), 181--224.

\bibitem[GT1]{GT1}
{W. T. Gan and S. Takeda},
{\em On the Howe duality conjecture in classical theta correspondence},
{\it Advances in the theory of automorphic forms and their $L$-functions}, 105--117, 
{\it Contemp. Math.}, {\bf664}, Amer. Math. Soc., Providence, RI, 2016.

\bibitem[GT2]{GT2}
{W. T. Gan and S. Takeda},
{\em A proof of the Howe duality conjecture},
{\it J. Amer. Math. Soc}. {\bf29} (2016), no.~2, 473--493.

\bibitem[HKK23]{HKK23} J. Haan, S. Kwon, Y. Kim, \emph{The local converse theorem for quasi-split $\O_{2n}$ and $\SO_{2n}$}, \href{https://arxiv.org/pdf/2301.12693v2.pdf}{https://arxiv.org/pdf/2301.12693v2.pdf}

\bibitem[HO13]{HO13}
V. Heiermann, E. Opdam,
{\em ON THE TEMPERED L-FUNCTIONS CONJECTURE}, 
{\it American Journal of Mathematics}, vol. 135, no. 3, (2013), pp. 777–799.

\bibitem[He]{He} G. Henniart, \emph{Caract\'{e}risation de la correspondance de Langlands locale par les facteurse de paires}, Invent. M., \textbf{113}, (1993), 339--350.

 \bibitem[JL18]{JL18} Herv\'{e} Jacquet and Baiying Liu,
   \emph{On the local converse theorem for $p$-adic $\GL_n$}, Amer. J. Math., \textbf{140}, (2018), no. 5, 1399--1422.

   \bibitem[Jng]{Jng}
D. Jiang, {\em On local $\gamma$-factors}. In: Arithmetic geometry and number theory, Series on Number Theory and Applications, 1, World Sci. Publ., Hackensack, NJ, (2006), 1-28. 

\bibitem[JSS]{JSS}, H. Jacquet, I.I. Piatetski-Shapiro, J.A. Shalika, \emph{Rankin--Selberg convolutions}, Amer. J. Math., \textbf{105}, (1983), 367--464.
   
\bibitem[JS03]{JS03}
Dihua Jiang and David Soudry, \emph{The local converse theorem for $\SO(2n+1)$ and applications}, Ann. of Math. (2), \textbf{157}, (2003), no.3, 743--806.
\bibitem[JS04]{JS04}
Dihua Jiang and David Soudry, \emph{Generic representations and local
  {L}anglands reciprocity law for {$p$}-adic {$\SO_{2n+1}$}},
  Contributions to automorphic forms, geometry, and number theory, Johns Hopkins Univ. Press, Baltimore, MD, 2004, pp.~457--519. 

  
\bibitem[JS07]{JS07}
D. Jiang and D. Soudry, 
\emph{On the genericity of cuspidal automorphic forms of $\mathrm{SO}_{2n+1}$, II}, 
Compos. Math. \textbf{143} (2007), no.3, 721--748.

  \bibitem[Jo]{Jo} Yeongseong Jo, \emph{The local converse theorem for odd special orthogonal and symplectic groups in positive characteristic}, 
\href{https://arxiv.org/pdf/2205.09004v5}{arXiv:2205.09004v5}

\bibitem[Kap15]{Kap15} Eyal Kaplan, \emph{Complementary results on the Rankin-Selberg gamma factors of
   classical groups}, J. Number Theory, \textbf{146}, (2015), 390--447.

\bibitem[Kim]{Kim} Henry H. Kim, \emph{Automorphic $L$-functions}, Lectures on automorphic $L$-functions, Fields Inst. Monogr., vol.20, Amer. Math. Soc., Providence, RI, (2004), 97--201.

\bibitem[Ku]{Ku}
{S. S. Kudla}, 
{\em Notes on the local theta correspondence}, \href{http://www.math.utoronto.ca/~skudla/castle.pdf}{http://www.math.utoronto.ca/~skudla/castle.pdf}.

\bibitem[LR05]{LR05}
E.~M. Lapid and S.~Rallis, 
\emph{On the local factors of representations of classical groups},
\newblock In J.~W. Cogdell, D.~Jiang, S.~S. Kudla, D.~Soudry, and R.~Stanton, editors, {\em Automorphic representations, ${L}$-functions and applications: progress and prospects}, 
{\it Ohio State Univ. Math. Res. Inst. Publ.} {\bf 11}, de Gruyter, Berlin, (2005) 309--359.

\bibitem[LH22]{LH22} B. Liu and  A. Hazeltine, \emph{On the local converse theorem for split SO(2n)}, (2022), submitted.

\bibitem[L15]{L15} 
L. Lomel\'{i}, \emph{The LS method for the classical groups in positive characteristic and the Riemann Hypothesis: Dedicated to Sofi'a Lomel\'{i}}, American Journal of Mathematics, vol. 137, no. 2, (2015), pp. 473–96.

\bibitem[L19]{L19} 
L. Lomelí, \emph{Rationality and holomorphy of Langlands–Shahidi L-functions over function fields}, Math. Z. \textbf{291}, (2019), 711–739.

\bibitem[M]{M}
K. Morimoto, \emph{On the irreducibility of global descents for even unitary groups and its applications}, Transactions of A.M.S., \textbf{370} (2018), 6245-6295.

\bibitem[MVW]{MVW}
C. Moeglin, M.-F. Vignéras and J.-L. Waldspurger, \emph{Correspondances de Howe sur un corps p-adique}, Springer Lecture Notes.

\bibitem[MS20]{MS20}
{G. Mui\'c and G. Savin},
{\em Symplectic-orthogonal theta lifts of generic descrete series}, {\it Duke Math. J}, {\bf 101}, (2000), 317--333.

\bibitem[PS]{PS}
D. Prasad, and R. Schulze-Pillot, 
\emph{Generalised form of a conjecture of Jacquet and a local consequence},
{\it Crelle},
{\bf no.616},
(2008) 219--236.

\bibitem[Rao93]{Rao93} R. Rao, \emph{On some explicit formulas in the theory of weil representation}, Pacific. J. Math. \textbf{157}, (1993), 335-371.

\bibitem[Sha90]{Sha90} F. Shahidi, \emph{A proof of Langlands' conjecture on Plancherel measures; complementary series for $p$-adic groups}, Ann. of. Maath. (2) \textbf{132} (1990), 273--330.

\bibitem[So93]{So93} D. Soudry, \emph{Rankin--Selberg convolutions for $\SO(2n+1)\times \GL(n)$: local theory}, Mem. Amer. Math. Soc. \textbf{105} (1993), 1--100.

\bibitem[Sz10]{Sz10} D. Szpruch,
\emph{The Langlands--Shahidi method for he metpletic group and applications, thesis, Tel-Aviv Univesity, Israel, (2010).}

\bibitem[Q17]{Q17} Qing Zhang, \emph{Stability of Rankin–Selberg gamma factors for $\Sp(2n)$,$\wt{\Sp}(2n)$ and $U(n,n)$}, Int. Journal of Number Theory, \textbf{13:9} (2017), 2393-2432.

  \bibitem[Q18]{Q18}
  Qing Zhang, \emph{A local converse theorem for $\Sp_{2r}$},
Math. Ann., \textbf{372}, (2018), no. 1-2, 451--488.

\bibitem[Q19]{Q19}
  Qing Zhang, \emph{A local converse theorem for ${\rm U}_{2r+1}$},
Trans. Amer. Math. Soc., \textbf{371}, (2019), no. 8, 5631--5654.

\bibitem[Wa90]{Wa90}
{J.-L. Waldspurger}, 
{\em Demonstration d'une conjecture de dualite de Howe dans le cas $p$-adique, $p\not= 2$}, 
{\it in Festschrift in honor of I. I. Piatetski-Shapiro on the occasion of his sixtieth birthday}, 
part $\mathrm{I}$ (Ramat Aviv, 1989), Israel Mathematical Conference Proceedings, vol. 2 (Weizmann, Jerusalem, 1990), 267--324.

\bibitem[YZ23]{YZ23} 
P. Yan and Q. Zhang, 
\emph{Product of Rankin-Selberg convolutions and a new proof of Jacquet’s local converse conjecture}, \href{https://arxiv.org/abs/2309.10445}{arXiv:2309.10445}

\end{thebibliography}
\end{document}